\documentclass{cimart}

\usepackage{float}

\newenvironment{smallarray}[1]
{\null\,\vcenter\bgroup\scriptsize
	\arraycolsep=.13885em
	\hbox\bgroup$\array{@{}#1@{}}}
{\endarray$\egroup\egroup\,\null}

\newcommand{\seqnum}[1]{\href{https://oeis.org/#1}{\rm \underline{#1}}}

\newcommand\Enn{\mathbb{N}}
\newcommand\Zee{\mathbb{Z}}
\newcommand\Que{\mathbb{Q}}

\newcommand\totk{{\rm Tot}^\kappa}
\newcommand\totn{{\rm Tot}^\nu}

\title{Proving Properties of
	$\varphi$-Representations with the {\tt Walnut} Theorem-Prover
}

\author{Jeffrey Shallit
}

\authorinfo{School of Computer Science,
	University of Waterloo, Canada}{shallit@uwaterloo.ca}

\abstract{We revisit a classic theorem of Frougny and Sakarovitch concerning automata for
	$\varphi$-representations, and show how to obtain it in a different
	and more computationally direct way.  Using it, we can
	find simple, induction-free proofs
	of existing results in the literature about these
	representations, in a uniform and straightforward
	manner.  In particular, we can easily and ``automatically''
	recover many of the results of recent papers of Dekking and Van Loon.
	We also obtain a number of new results on $\varphi$-representations.
}



\VOLUME{33}
\ISSUE{2}
\YEAR{2025}
\NUMBER{3}
\DOI{https://doi.org/10.46298/cm.12627}

\begin{document}
	
	% Here is where the main text should be typed:
	
	\vskip .1in
	\centerline{\hfill\it In honor of the 75th birthday of Christiane Frougny}
	
	\section{Introduction}
	\label{one}
	
	In 1957, George Bergman \cite{Bergman:1957} observed that one can
	expand a non-negative
	real number $z$ in base $\varphi = (1+\sqrt{5})/2$, in the
	sense that there exist ``digits'' $a_i \in \{0,1\}$ such that
	\begin{equation}
		z = \sum_{-\infty \leq i < r} a_i \varphi^i .
		\label{fund}
	\end{equation}
	In analogy with base-$b$ representation, we can write
	such a base-$\varphi$ representation as a string of digits
	$a_{r-1} \cdots a_0 \, . \, a_1 a_2 a_3 \cdots$ 
	where the non-negative and negative powers
	are separated by the analogue of a ``decimal point''.  
	We call $a_{r-1} \cdots a_0$ the {\it left part} and
	$ . \, a_1 a_2 a_3 \cdots$ the {\it right part} of
	the $\varphi$-representation.
	The left part is always finite, while the right part
	may be finite or infinite.
	These parts are analogous to the more familiar
	integer and fractional parts of decimal representations of the
	positive reals.
	
	If Equation~\eqref{fund} holds, then we write
	$$ z = [a_{r-1} \cdots a_0 \, . \, a_1 a_2 a_3 \cdots]_\varphi.$$
	For example,
	$ 2 = [10.01]_\varphi = [1.11]_\varphi $
	and
	$ \frac{1}{2} = [.010010010010 \cdots]_\varphi$.
	
	Furthermore, such a representation $x.y$ as a string of binary digits with
	a decimal point is essentially unique
	if one imposes the very natural
	restriction that $a_i a_{i+1} \not= 1$ for all $i$
	(or, equivalently, if the string $xy$ contains no occurrence of the
	block $11$), and
	furthermore that if the expansion is infinite, it cannot end in
	$010101\cdots$.  
	Throughout the paper,
	we adopt the convention that in expansions of the form $[x.y]_\varphi$,
	we disregard leading zeros in $x$ and trailing zeros in $y$. Thus
	$10.01$ and $010.0100$ are regarded as ``essentially'' the same.
	This convention is extremely useful, as sometimes our automata will
	need to pad multiple inputs with zeros to ensure they all
	have the same length.

	We call a $\varphi$-representation obeying these rules
	{\it canonical\/}, and write it as $(x)_\varphi$.
	Bergman \cite{Bergman:1957} proved
	that the canonical $\varphi$-representation of a natural number
	is finite, and Table~\ref{tab1} gives the 
	canonical expansions of the first few positive integers.
	
	\begin{table}[H]
		\begin{center}
			\begin{tabular}{c|cccccccccc}
				$n$ & 1 & 2 & 3 & 4 & 5 & 6 & 7 \\
				\hline
				$(n)_\varphi$ & 1. & 10.01 & 100.01 & 101.01 & 1000.1001 & 1010.0001
				& 10000.0001
			\end{tabular}
		\end{center}
		\caption{First few canonical expansions.}
		\label{tab1}
	\end{table}
	
	The left and right parts of the canonical $\varphi$-representation of $n$
	can be found in sequences \seqnum{A105424} and \seqnum{A341722}
	in the {\it On-Line Encyclopedia of Integer Sequences} (OEIS)
	\cite{Sloane:2023}.
	
	Bergman's ideas were generalized to positive real numbers
	$\beta$ by R\'enyi \cite{Renyi:1957}, who
	called such expansions {\it $\beta$-expansions}.  
	Later, fundamental work characterizing such
	expansions was done by Parry \cite{Parry:1960}.
	
	Several researchers, including Bertrand-Mathis
	and Frougny, have found deep and interesting connections
	between these expansions and finite automata when $\beta$ is an
	algebraic number.  See, for example, 
	\cite{Bertrand-Mathis:1989,Frougny:1989,Frougny:1991a,Frougny:1992a,Frougny:1992b,Frougny:1992c,Frougny:1992d,Berend&Frougny:1994,Frougny:2003}.
	
	In a very interesting article,
	Frougny and Sakarovitch \cite{Frougny&Sakarovitch:1999} proved that
	the canonical $\varphi$-representation of a non-negative integer can be
	computed by a finite automaton $A$, in the following sense:
	$A$ takes three inputs {\it in parallel:}
	$n$ represented in the Zeckendorf numeration system (see below),
	and binary strings $x$ and $y$,
	and accepts if and only if $(n)_\varphi = x.y^R$.  Here, $y^R$ denotes
	the reversal of the string $y$. This representation is called
	``folded''; it is
	useful in part because the coefficients corresponding to small (in
	absolute value) powers of $\varphi$ are grouped together when
	read by an automaton.  For example, the entry of Table~\ref{tab1}
	corresponding to $n= 6$ would correspond to the accepted input
	$ [1,1,1][0,0,0][0,1,0][1,0,0]$:   the first entries of each
	triple spell out $1001$, which is $6$ in Zeckendorf representation;
	the second entries spell out $1010$ and the third entries
	spell out $1000$.
	However, their article is rather challenging to read and
	they did not explicitly present this automaton.
	
	The first goal of this paper is to explain how the Frougny-Sakarovitch
	automaton can be, rather easily,
	computed explicitly using existing software; namely,
	the free software {\tt Walnut} initially developed by Hamoon Mousavi
	\cite{Mousavi:2021}.  This software can prove or disprove theorems
	about automatic sequences and their generalizations, using a
	decision procedure for a certain extension of Presburger
	arithmetic \cite{Shallit:2023}.  It suffices to express the desired
	assertions in first-order logic.   Furthermore, if the particular
	logical statement has free variables, {\tt Walnut} will compute a
	finite automaton accepting precisely those values of the free variables making
	the statement true.   Finally, if a formula $F$ has two or more free
	variables, say $n$ and $x_1,x_2,\ldots x_t$, {\tt Walnut}
	can compute the number of $t$-tuples $(x_1,\ldots, x_t)$ such that
	$F(n,x_1,\ldots, x_t)$ evaluates to {\tt TRUE}.
	
	Once we have this automaton, we can use it to easily reprove
	existing results from the literature in a straightforward and uniform
	manner.   For example, we reprove some results
	recently considered by Dekking and Van Loon \cite{Dekking&van.Loon:2023}.
	They called a finite $\varphi$-representation (not necessarily canonical)
	of an integer $n$ a
	{\it Knott expansion} if it does not end in $011$, and they developed
	a rather complicated method for computing
	the number of different Knott expansions of $n$.  We will see that 
	the Frougny-Sakarovitch automaton
	allows us to enumerate Knott expansions, and thereby recover
	the results of Dekking and Van Loon in a purely ``automatic'' fashion,
	without any tedious inductions.
	Furthermore, by simply changing the conditions we impose on the
	form of the expansion, we can easily and ``automatically'' enumerate
	other types of expansions.  The second goal of the paper is to
	illustrate these techniques and obtain a number of new results.
	
	The paper is organized as follows.   In Section~\ref{two} we construct
	the Frougny-Sakarovitch automaton.
	In Sections~\ref{five}--\ref{eight} we prove
	various old and new results about $\varphi$-representations; one
	new result is Theorem~\ref{gerdemann}, which  proves a 2012 conjecture of
	Dale Gerdemann.
	In Section~\ref{three}
	we discuss Knott expansions;
	in Section~\ref{four} we discuss
	a different type of expansion, called a ``natural expansion'', introduced
	by Dekking and Van Loon; and in Section~\ref{fiveb} we discuss
	another kind of expansion, called DVL-expansions.   Finally, in Section~\ref{ten} we discuss more general
	expansions for algebraic integers.
	
	\section{The Frougny-Sakarovitch automata}
	\label{two}
	
	We will need two different representations of 
	integers:  the Zeckendorf system and the negaFibonacci system.
	
	Let the Fibonacci numbers be defined by 
	the recurrence $F_n = F_{n-1} + F_{n-2}$ and initial conditions
	$F_0 = 0$ and $F_1 = 1$.   Note that this uniquely defines
	$F_n$ for {\it all\/} integers $n$ (even negative integers).
	
	In the well-known Zeckendorf system \cite{Lekkerkerker:1952,Zeckendorf:1972},
	we write a non-negative integer $n$ as a sum
	$n = \sum_{2 \leq i \leq t} a_i F_i$ with $a_i \in \{0,1\}$.
	We abbreviate this by the expression $[x]_F$, where
	$x = a_t a_{t-1} \cdots a_2$.
	This representation
	is unique if we impose the condition $a_i a_{i+1} \not= 1$;
	we let $(n)_F$ denote the canonical representation obeying this condition.
	Thus $(43)_F = 10010001$.
	
	It will also be useful to be able to represent negative integers.
	In the negaFibonacci system \cite{Bunder:1992,Shallit&Shan&Yang:2023},
	we write an integer $n$ (positive, negative, or zero) as a sum
	$n = \sum_{1 \leq i \leq t} a_i F_{-i}$, where again $a_i \in \{0,1\}$.
	We abbreviate this by the expression $[x]_{-F}$, where
	$x = a_t \cdots a_1$. Again, this representation
	is unique if we impose the condition $a_i a_{i+1} \not= 1$;
	we let $(n)_{-F}$ denote the canonical representation obeying this condition.
	Thus $(43)_{-F} = 101001010$.  
	
	(An interesting
	alternative to the negaFibonacci system has recently been proposed
	by Labb\'e and Lep\v{s}ov\'a; see \cite{Labbe&Lepsova:2023}.)
	
	We start by constructing the first Frougny-Sakarovitch automaton.
	It has three inputs over $\{0,1\}$ that are read in
	parallel:  $w, x, $ and $y$.   Here 
	\begin{itemize}
		\item $w$ is the (canonical) Zeckendorf representation of some integer $n$,
		\item $x$ is the left part of a $\varphi$-representation
		\item $y^R$ is the right part of a $\varphi$-representation
	\end{itemize}
	and the automaton accepts if and only if $n = [x.y]_\varphi$.
	Notice that for this automaton
	{\it we impose no other conditions\/} on $x$ and $y$; they are
	simply arbitrary binary strings.
	However, $w$ is assumed to contain no occurrence of $11$.
	All three inputs are allowed to have any number of leading zeros, which
	permits the inputs to be of the same length so they can be
	read simultaneously in parallel.
	
	The basic idea is very simple.   As Bergman \cite{Bergman:1957}
	noted (and easily proved by induction), we have
	$\varphi^k = F_k \varphi + F_{k-1}$ for all integers $k$.  It now follows that
	\begin{align}
		\sum_{0 \leq i < r} a_i \varphi^i &= \sum_{0 \leq i < r}
		a_i (F_i \varphi + F_{i-1}) \nonumber \\
		&= \left( \sum_{0 \leq i < r} a_i F_i \right) \varphi  +
		\left(\sum_{0 \leq i < r} a_i F_{i-1} \right) \nonumber \\
		&= \left( a_0 F_0 + a_1 F_1 + \sum_{2 \leq i < r} a_i F_i \right) \varphi
		+ a_0 F_{-1} + a_1 F_0 + a_2 F_1 +
		\sum_{2 \leq j < r-1} a_{j+1} F_j \nonumber \\
		&= (a_1 + [a_r \cdots a_2]_F) \varphi + [a_r \cdots a_3]_F + a_0 + a_2.
		\label{berg1}
	\end{align}
	Similarly,
	\begin{align}
		\sum_{1 \leq i \leq s} b_i \varphi^{-i} &= 
		\sum_{1 \leq i \leq s} b_i (F_{-i} \varphi + F_{-i-1}) \nonumber\\
		&= \left( \sum_{1 \leq i \leq s} b_i F_{-i} \right) \varphi + 
		\sum_{1 \leq i \leq s} b_i F_{-i-1} \nonumber \\
		&= [b_s \cdots b_1]_{-F} \, \varphi  + [b_s \cdots b_1 0]_{-F}  . \label{berg2}
	\end{align}
	Thus, the left part of a $\varphi$-representation is easily expressed
	as a linear combination $c\varphi + d$, in terms of shifted
	Zeckendorf representations, and the right part of a $\varphi$-representation
	is easily expressed as a linear combination $c'\varphi + d'$,
	in terms of shifted negaFibonacci representations.  Then
	$n = (c+c') \varphi + d + d'$ implies that $n = [a_{r-1} \cdots a_1 a_0.b_1 b_2 \cdots b_s]$ if and only if
	\begin{align}
		c' &= -c \label{cond1} \\
		n &= d+ d' \label{cond2}
	\end{align}
	both hold.
	
	These are conditions that
	can be expressed in first-order logic, and hence we can build an
	automaton to check them.
	In {\tt Walnut} we can
	build this automaton from smaller, easily-digestible pieces,
	as we now describe.  
	
	It may be helpful to have a short summary of {\tt Walnut} syntax.
	\begin{itemize}
		\item {\tt reg} defines a regular expression
		\item {\tt def} defines an automaton based on a logical formula
		\item {\tt eval} evaluates a formula with no free variables and
		returns {\tt TRUE} or {\tt FALSE}
		\item {\tt E} represents the existential quantifier $\exists$;
		{\tt A} represents the universal quantifier $\forall$
		\item {\tt \&} is logical {\tt AND}, {\tt |} is logical {\tt OR},
		{\tt =>} is logical implication; {\tt <=>} is logical {\tt IFF};
		{\tt \char'176} is logical {\tt NOT}
		\item {\tt ?msd\_fib} instructs {\tt Walnut} that numbers
		should be expressed in Zeckendorf representation;
		{\tt ?msd\_neg\_fib} does the same thing for negaFibonacci
		representation.
	\end{itemize}

	\subsection{Zeckendorf normalizer}
	
	We need a ``normalizer'' for Zeckendorf expansions; it takes
	$x$ and $y$ as inputs, with $x$ an arbitrary binary string and
	$y$ a canonical Zeckendorf expansion of some integer $n$,
	and accepts if and only if $n = [x]_F$.
	Such an automaton can be found, for example, in 
	\cite{Berstel:2001,Shallit:2021}.
	It has $5$ states (or $4$ if one omits the dead state).
	We call it {\tt fibnorm}.
	\begin{figure}[htb]
		\begin{center}
			\includegraphics[width=3.5in]{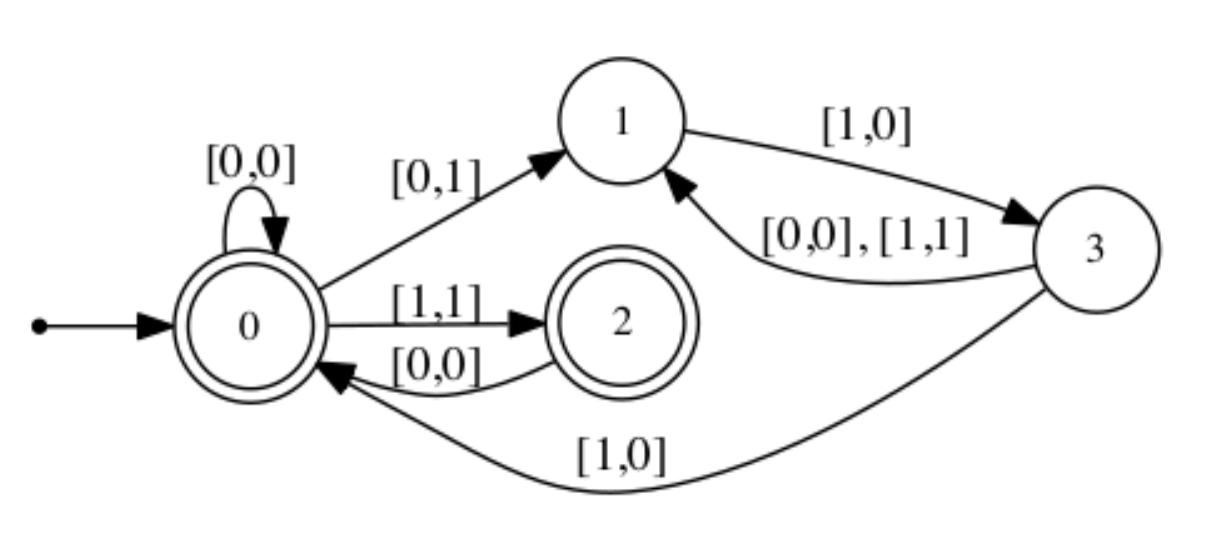}
		\end{center}
		\caption{The automaton {\tt fibnorm}.}
		\label{fibnormaut}
	\end{figure}
	
	\subsection{NegaFibonacci normalizer}
	We also need the analogous ``normalizer'' for negaFibonacci expansions; it takes
	$x$ and $y$ as inputs, with $x$ an arbitrary binary string and
	$y$ a canonical negaFibonacci expansion of some integer $n$,
	and accepts if and only if $n = [x]_{-F}$.
	Such an automaton has $5$ states (or $4$ if one omits the dead state).
	We call it {\tt negfibnorm}.  Correctness is left to the reader.
	\begin{figure}[H]
		\begin{center}
			\includegraphics[width=4in]{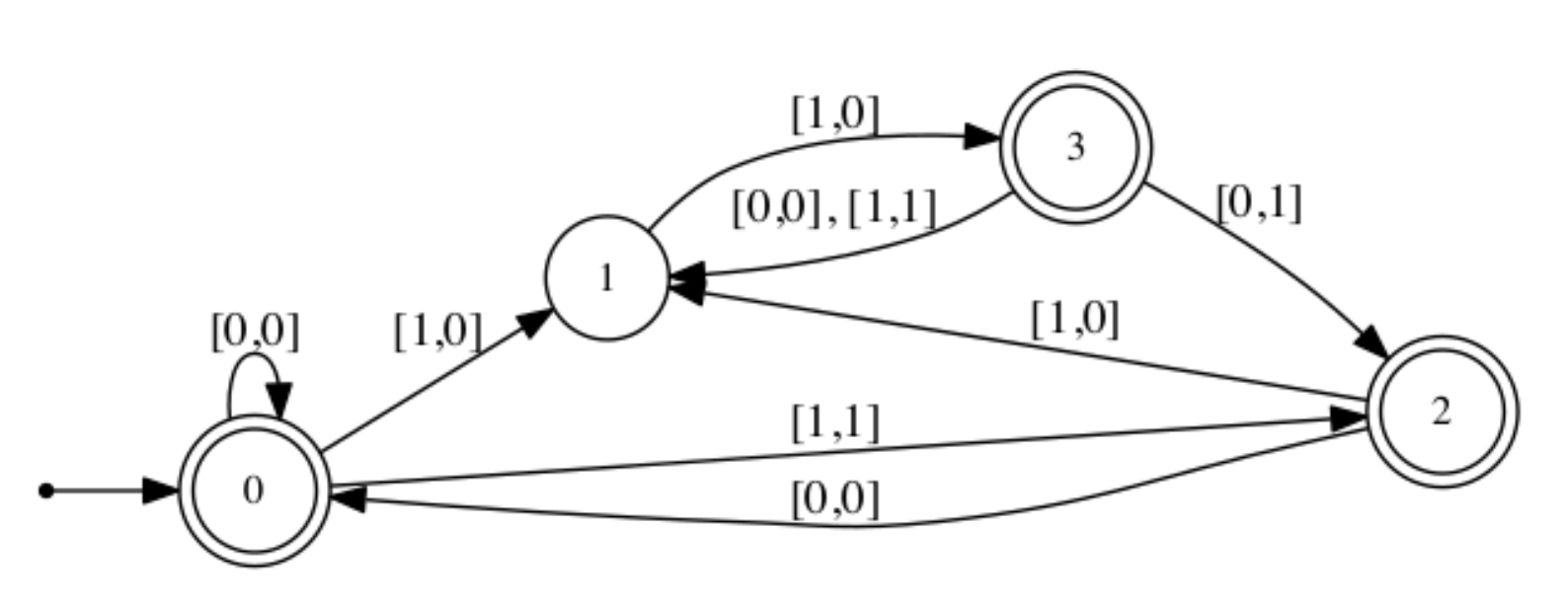}
		\end{center}
		\caption{The automaton {\tt negfibnorm}.}
		\label{negfibnormaut}
	\end{figure}

	\subsection{Shifters}
	As we have seen in Equations~\eqref{berg1} and \eqref{berg2}, we need
	to shift some representations left or right.  We do this with the
	following automata:
	\begin{itemize}
		\item {\tt shiftl} takes arguments $x$ and $y$, and accepts if $y$
		is $x$ shifted to the left (and $0$ inserted at the right).
		\item {\tt shiftr} takes arguments $x$ and $y$, and accepts if
		$y$ equals the string $x$ shifted to the right (and the least
		significant digit disappears).
	\end{itemize}
	We can easily build these with regular expressions as follows:
	\begin{verbatim}
		reg shiftl {0,1} {0,1} "([0,0]|[0,1][1,1]*[1,0])*":
		reg shiftr {0,1} {0,1} "([0,0]|[1,0][1,1]*[0,1])*(()|[1,0][1,1]*)":
	\end{verbatim}
	
	\subsection{Last bits}
	Again, as we have seen in Equations~\eqref{berg1} and \eqref{berg2}, we need
	the ability to extract the last, second-to-last, or third-to-last
	bits of a string.  These can easily be specified with regular
	expressions.  We create
	three:  {\tt lstbit1}, {\tt lstbit2}, and {\tt lstbit3}.
	Here {\tt lstbitn} %$n$
	takes arguments $x$ and $y$, where $x \in \{0,1\}^*$ and $y$
	is the representation of either $0$ or $1$, and accepts
	if $y$ is the $n$'th-last bit of $x$ (or $0$ if $|x|<n$).
	\begin{verbatim}
		reg lstbit1 {0,1} msd_fib "()|(([0,0]|[1,0])*([0,0]|[1,1]))":
		reg lstbit2 {0,1} msd_fib "()|[0,0]|[1,0]|(([0,0]|[1,0])*
		(([0,0]([0,0]|[1,0]))|([1,0]([0,1]|[1,1]))))":
		reg lstbit3 {0,1} msd_fib "()|[0,0]|[1,0]|(([0,0]|[1,0])
		([0,0]|[1,0]))|(([0,0]|[1,0])*(([0,0]([0,0]|[1,0])
		([0,0]|[1,0]))|([1,0]([0,0]|[1,0])([0,1]|[1,1]))))":
	\end{verbatim}
	
	\subsection{Converting from negaFibonacci to Zeckendorf}
	Finally, we will need to be able to convert negaFibonacci representation
	to Zeckendorf representation.   We need two automata:
	\begin{itemize}
		\item {\tt fibnegfib} takes two arguments $x$ and $y$, where
		$x = (n)_F$ for some $n \geq 0$ and
		$y = (m)_{-F}$, and
		accepts if $n = m$.  It has $12$ states.
		
		\item {\tt fibnegfib2} takes two arguments $x$ and $y$, where
		$x = (n)_F$ for some $n \geq 0$ and
		$y = (m)_{-F}$, and
		accepts if $n = -m$.  It has $9$ states.
	\end{itemize}
	
	\subsection{Constructing the first Frougny-Sakarovitch automaton}
	
	We now have all the pieces we need to construct the first
	Frougny-Sakarovitch automaton directly from the formulas
	\eqref{berg1} and \eqref{berg2}.  We do this with the following
	{\tt Walnut} code:
	\begin{verbatim}
		def phipartleft "?msd_fib Er,s,y,b $shiftr(x,r) & $shiftr(r,s) &
		$fibnorm(s,y) & $lstbit2(x,b) & z=y+b":
		def intpartleft "?msd_fib Er,s,t,y,b,c $shiftr(x,r) & $shiftr(r,s) &
		$shiftr(s,t) & $fibnorm(t,y) & $lstbit1(x,b) & $lstbit3(x,c) &
		z=y+b+c":
		def phipartright "?msd_neg_fib $negfibnorm(x,z)":
		def intpartright "?msd_neg_fib Er $shiftl(x,r) & $negfibnorm(r,z)":
		
		def frougny1 "?msd_fib Et1,t2 $phipartleft(x,t1) &
		$phipartright(y,?msd_neg_fib t2) & $fibnegfib2(t1,t2)":
		# accepts if the phi part of inputs x and y sums to 0
		
		def frougny2 "?msd_fib Et1, t2, x2 $intpartleft(x,t1) 
		& $intpartright(y,?msd_neg_fib t2) & (((?msd_neg_fib t2<0) &
		$fibnegfib2(x2,?msd_neg_fib t2) & ?msd_fib t1=n+x2)|
		(((?msd_neg_fib t2>=0) & $fibnegfib(x2,?msd_neg_fib t2) &
		?msd_fib n=t1+x2)))":
		# accepts if the integer parts of inputs x and y sum to n
		
		def frougny3 "?msd_fib $frougny1(x,y) & $frougny2(n,x,y)":
		fixleadzero frougny frougny3:
	\end{verbatim}
	Here {\tt phipartleft} (resp., {\tt intpartleft}) computes the
	coefficient of $\varphi$ (resp., the coefficient of $1$)
	in the expression \eqref{berg1}.
	Similarly,
	{\tt phipartright} (resp., {\tt intpartright}) does the same
	thing for \eqref{berg2}.   The command {\tt fixleadzero} allows one
	to remove useless leading zeroes from the resulting automaton.
	
	Next, {\tt frougny1} checks the condition \eqref{cond1} and
	{\tt frougny2} checks the condition \eqref{cond2}.   Finally,
	{\tt frougny} checks the conjunction of the two conditions.
	The resulting automaton {\tt frougny} has 116 states.
	
	If we want a simpler automaton, namely, where $x.y^R$ must
	be a canonical representation (no occurrences of $11$ allowed),
	we can intersect the automaton
	{\tt frougny} with a simple automaton imposing this condition 
	on $x$ and $y$.  This gives us a second Frougny-Sakarovitch
	automaton {\tt saka} with only $39$ states.
	
	\begin{verbatim}
		def saka2 "?msd_fib $frougny(n,?msd_fib x,?msd_fib y) & $no11xy(x,y)":
		fixleadzero saka saka2:
	\end{verbatim}
	
	With this automaton we can easily prove the following characterization of
	the canonical $\varphi$-representation of $F_n$, which is
	the content of both
	\cite[Corollary 6]{Frougny&Sakarovitch:1999} and
	\cite[Proposition~3.1]{Dekking&van.Loon:2023}.
	
	\begin{proposition}
		We have
		\begin{itemize}
			\item $(F_{4i})_\varphi = 100(0100)^{i-1}.(0100)^{i-1}1$, $i \geq 1$;
			\item $(F_{4i+1})_\varphi = 1000(1000)^{i-1}.(1000)^{i-1}1001$, $i \geq 1$;
			\item $(F_{4i+2})_\varphi = 1(0001)^i.(0001)^i$, $i \geq 0$;
			\item $(F_{4i+3})_\varphi = 10(0010)^i.(0010)^i01$, $i \geq 0$.
		\end{itemize}
	\end{proposition}
	
	\begin{proof}
		We apply the automaton {\tt saka} to the Fibonacci numbers $F_n$ for $n\geq 2$:
		\begin{verbatim}
			reg isfib msd_fib "0*10*":
			def prop31a "?msd_fib $saka(n,x,y) & $isfib(n)":
			fixleadzero prop31 prop31a:
		\end{verbatim}
		The resulting automaton is depicted in Figure~\ref{aut1}.
		\begin{figure}[h]
			\begin{center}
				\includegraphics[width=6in]{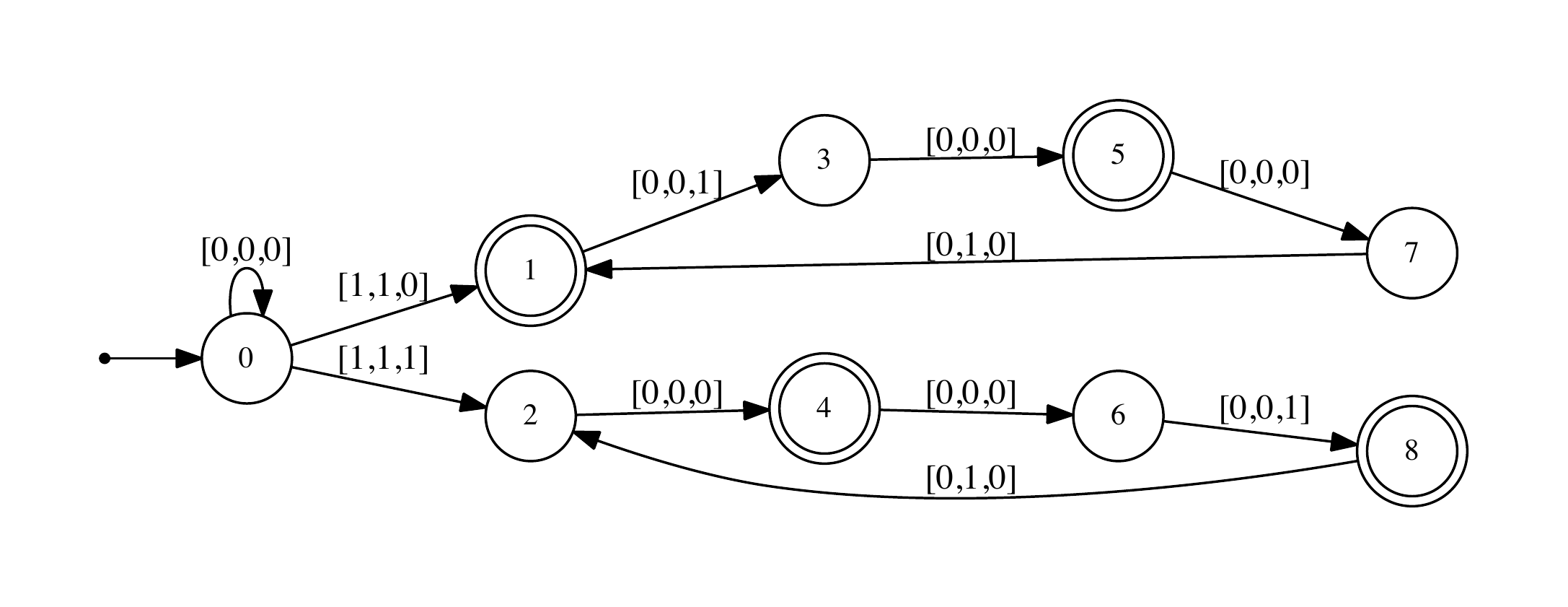}
			\end{center}
			\caption{Canonical $\varphi$-representation of $F_n$.}
			\label{aut1}
		\end{figure}
		
		Inspection of the automaton shows that, ignoring leading zeros,
		there are essentially four different acceptance paths:
		\begin{itemize}
			\item $[1,1,0][0,0,1][0,0,0]([0,0,0][0,1,0][0,0,1][0,0,0])^*$.
			\item $[1,1,1][0,0,0][0,0,0][0,0,1]([0,1,0][0,0,0][0,0,0][0,0,1])^*$;
			\item $[1,1,0]([0,0,1][0,0,0][0,0,0][0,1,0])^*$;
			\item $[1,1,1][0,0,0]([0,0,0][0,0,1][0,1,0][0,0,0])^*$;
		\end{itemize}
		These correspond to the cases $n \equiv 0,1,2,3$ (mod $4$),
		respectively, and prove the result.
	\end{proof}
	
	We can also see from Figure~\ref{aut1} a partial explanation for
	the periodicity (mod 4) in the Frougny-Sakarovitch construction.
	
	Exactly the same techniques can be used to prove a similar
	result for the Lucas numbers, defined by
	$L_0 = 2$, $L_1 = 1$, and $L_n = L_{n-1} + L_{n-2}$.  We omit the details.
	\begin{theorem}
		We have
		\begin{itemize}
			\item $(L_{2i})_\varphi = 10^{2i}.0^{2i-1}1$, $i \geq 1$;
			\item $(L_{2i+1})_\varphi = (10)^i 1.(01)^i$, $i \geq 0$.
		\end{itemize}
	\end{theorem}
	
	One very useful feature of the automaton {\tt saka} is we can
	use it to completely characterize the folded representations of
	all natural numbers.  
	\begin{theorem}
		The string $a_{r-1} \cdots a_1 a_0.a_{-1} a_{-2} \cdots a_{-r}$
		is the $\varphi$-representation of an integer $n \geq 0$
		if and only if
		$$[a_{r-1}, a_{-r}] [a_{r-2},a_{1-r}] \cdots [a_1,a_{-2}] [a_0, a_{-1}]$$
		is accepted by the automaton in Figure~\ref{beta10}.
		\begin{figure}[htb]
			\begin{center}
				\includegraphics[width=6.3in]{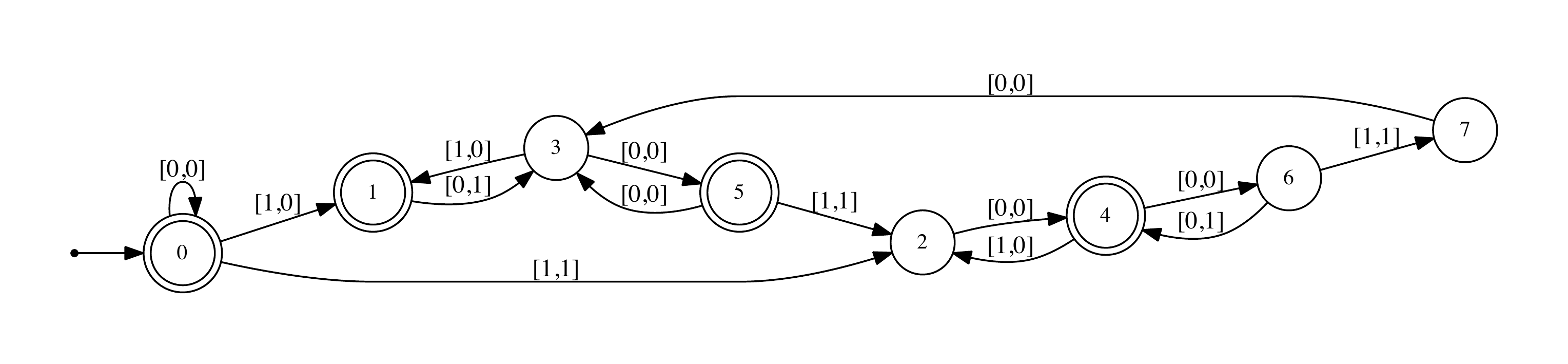}
			\end{center}
			\caption{Automaton accepting the folded representations of all $n \geq 0$.}
			\label{beta10}
		\end{figure}
	\end{theorem}
	\begin{proof}
		We use the {\tt Walnut} code
		\begin{verbatim}
			def berg "?msd_fib En $saka(n,x,y)":
		\end{verbatim}
	\end{proof}

	Let us now turn to characterizing the right part of canonical
	$\varphi$-representations.   Dekking \cite[Theorem~7.3]{Dekking:2023}
	proved the following result:
	\begin{theorem}
		If a binary string is of even length, ends in a $1$, and
		contains no $11$, then it appears as the right part of the
		$\varphi$-representation of some natural number $n$.
	\end{theorem}
	
	\begin{proof}
		To prove this, we make a {\tt Walnut} regular expression
		for strings of even length ending in a $1$.
		Some care is needed here, for two reasons:  first, the representation
		of the right part is ``folded'', so we're really talking about
		beginning with a $1$ instead of ending.  Second, our automata
		may have leading zeros in the (reversed) representation of the right part,
		for the reasons mentioned previously.  Keeping these two points in mind,
		we use the following {\tt Walnut} code:
		\begin{verbatim}
			reg has11 {0,1} "(0|1)*11(0|1)*":
			def rightp2 "?msd_fib En,x $saka(n,x,y)":
			fixleadzero rightp rightp2:
			reg evenl {0,1} "(0*)|(0*1(0|1)((0|1)(0|1))*)":
			eval thm73 "?msd_fib Ay $rightp(y) <=> ($evenl(y) & ~$has11(y))":
		\end{verbatim}
		And {\tt Walnut} returns {\tt TRUE}.
	\end{proof}
	
	In a similar way, we can prove a companion theorem for the left part
	of canonical $\varphi$-representations.
	This appears to be new.
	\begin{theorem}
		A binary string $x$ appears as the left part of a $\varphi$-representation
		of some $n$ if and only if $x$ has no $11$ and does not have a suffix
		of the form $1(00)^i 1$.
	\end{theorem}
	
	\begin{proof}
		We use the following {\tt Walnut} code:
		\begin{verbatim}
			def leftp2 "?msd_fib En,y $saka(n,x,y)":
			fixleadzero leftp leftp2:
			reg suff {0,1} "(0|1)*1(00)*1":
			eval claim "?msd_fib Ax $leftp(x) <=> ((~$suff(x))&(~$has11(x)))":
		\end{verbatim}
		And {\tt Walnut} returns {\tt TRUE}.
	\end{proof}
	
	It is easy to see that every
	positive integer has infinitely many (possibly non-canonical)
	$\varphi$-representations; for example, we may
	write 
	$$2 = [10.01]_\varphi = [10.0011]_\varphi = [10.001011]_\varphi = \cdots .$$
	However, there can only be finitely many distinct left parts among
	all of these.  How many are there?
	We can use {\tt Walnut} to ``automatically'' 
	compute a special kind of formula for
	this quantity, called a linear representation.
	
	A {\it linear representation\/} for a function $f$ from
	$\Sigma^*$ to $\Enn$ is a triple
	$(v, \gamma, w)$, where $v$ is a $t$-element row vector,
	$\gamma$ is a $t\times t$ matrix-valued morphism,
	and $w$ is a $t$-element column vector, such that
	$f(x) = v \gamma(x) w$.  The {\it rank\/} of such a representation
	is defined to be $t$.  Two linear representations are said
	to be equivalent if they compute the same function of $x$.  There is an
	algorithm that, given a linear representation, finds an equivalent minimal
	one:  that is, one of minimal rank \cite[Chapter~2]{Berstel&Reutenauer:2011}.
	A nice feature of linear representations is that they allow efficient
	computation of the corresponding function.
	For more information about
	linear representations, see
	\cite{Berstel&Reutenauer:2011}.  
	
	Let us find a linear
	representation for $p(n)$,
	the number of distinct left parts of $\varphi$-representations
	for $n$.  We can do this with the following {\tt Walnut} command:
	\begin{verbatim}
		def numleft n "?msd_fib Ey $frougny(n,x,y)":
	\end{verbatim}
	This gives us a linear representation of rank $112$ to
	compute $p(n)$, which can be minimized to a linear representation of
	rank $28$.
	The first few terms of this sequence are given in Table~\ref{tableft}.
	\begin{table}[htb]
		\begin{center}
			\begin{tabular}{c|ccccccccccccccccccccc}
				$n$ & 0 & 1 & 2 & 3 & 4 & 5 & 6 & 7 & 8 & 9 & 10 & 11 & 12 & 13 & 14 & 15 & 16 &17 & 18 & 19 \\
				\hline
				$p(n)$ & 1&2&2&3&3&3&3&4&5&4&5&4&4&5&6&6&5&4&5&7
			\end{tabular}
		\end{center}
		\caption{First few values of $p(n)$.}
		\label{tableft}
	\end{table} 
	This is sequence \seqnum{A362970} in the OEIS.
	
	\section{Base-\texorpdfstring{$\varphi$}{φ} sum of digits}
	\label{five}
	
	Dekking \cite{Dekking:2021} studied the sum-of-digits function for
	base-$\varphi$ representations.  
	Denoting by $n = \sum_{L \leq i \leq R} d_i \varphi^i$ the canonical
	base-$\varphi$ representation,
	then we define $s(n) = \sum_{L \leq i \leq R} d_i$.
	The sequence $s(n)$ is 
	sequence \seqnum{A055778} in the OEIS.
	
	Let us compute a linear representation for $s(n)$.
	The idea
	is to define an automaton with two binary inputs $x,y$ that accepts precisely
	if $y$ has a single $1$ in a position matching a $1$ in $x$.   Then
	the number of pairs $(x,y)$ accepted corresponds to the number of $1$'s
	in $x$.  We can then find linear representations for the
	left and right parts of $(n)_\varphi$.
	\begin{verbatim}
		reg match1 {0,1} {0,1} "([0,0]|[1,0])*[1,1]([0,0]|[1,0])*":
		def countl "?msd_fib Ex,y $saka(n,x,y) & $match1(x,t)":
		def countr "?msd_fib Ex,y $saka(n,x,y) & $match1(y,u)":
		def sl n "?msd_fib $countl(n,t)":
		def sr n "?msd_fib $countr(n,u)":
	\end{verbatim}
	This gives a linear representation of rank $99$ for $s_L$ and 
	one of rank $82$ for $s_R$.  When minimized, we get a representation
	of rank $19$ for $s_L$ and one of rank $21$ for $s_R$.  We can then
	find a linear representation of rank $40$ for $s$ by summing these
	two linear representations.   This can be minimized to a
	linear representation for $s$ of rank $21$.
	
	We now show how to apply our automaton {\tt saka} to solve an open problem
	due to Dale Gerdemann in 2012, and stated in the text for 
	sequence \seqnum{A055778} in the OEIS.  Once again, the proof is
	more or less purely computational, but here we need some additional
	computational tools.
	
	\begin{theorem}
		The sum of the
		digits for $n$ in Zeckendorf representation is always at most the
		sum of the digits for $n$ in canonical base-$\varphi$ representation.
		\label{gerdemann}
	\end{theorem}
	
	\begin{proof}
		Let $|x|_a$ denote the number of occurrences of the symbol $a$
		in the string $x$.
		It is easy, using the techniques above, to obtain a linear
		representation of rank $23$ for the difference $|(n)_\varphi|_1 - |(n)_F|_1$.
		However, there is no algorithm for determining whether a given
		linear representation computes a function that is always non-negative,
		so this linear representation does not seem to be useful to solve 
		the problem.  Instead, we use a different technique.
		
		Recall that the automaton ${\tt saka}$ takes three inputs in
		parallel---$n$, $x$, and $y$---and accepts if
		$x.y^R$ is the canonical $\varphi$-representation for $n$.
		The transitions of this automaton are therefore
		of the form $[a,b,c]$ where $a,b,c \in \{0,1\}$.  Suppose an input
		$a_1 \cdots a_t$, $b_1 \cdots b_t$, $c_1 \cdots c_t$ is accepted.
		Then $|(n)_\varphi|_1 - |(n)_F|_1 = \sum_{1 \leq i \leq t} (b_i+c_i)-a_i$.
		
		Therefore, we can take the automaton for ${\tt saka}$ and form
		a directed graph $G$ from it by replacing the transition on
		the triple $[a,b,c]$ with a weight labelled $b+c-a$.
		The assertion that $|(n)_\varphi|_1 - |(n)_F|_1$ is always non-negative
		then follows from the following stronger
		claim:
		\begin{equation}
			\text{Every path from the initial state to any state has
				a non-negative sum of weights.} \label{sclaim}
		\end{equation}
		But this is precisely the minimum-weight path problem solved by the
		Bellman-Ford algorithm.  We ran this algorithm on $G$ and
		verified that assertion~\eqref{sclaim} holds.
	\end{proof}
	
	Next, one can consider those $n$ for which $|(n)_F|_1 = |(n)_\varphi|_1$.
	These correspond to the paths of weight $0$ in the graph $G$ just discussed.
	One can easily verify, by explicit enumeration, that all simple cycles of
	$G$ (i.e., no repeated vertices except at the beginning and end of the cycle)
	have weight $\geq 0$, and furthermore one can determine all those of
	weight $0$ (there are $26$ of them).
	
	Let $P$ be an accepting path of weight $0$.   If $P$ contains a cycle
	$C$, by above the weight of $C$ is non-negative, by removing $C$,
	we get an accepting path of $P'$ of weight $\leq 0$.
	But there are no negative-weight
	accepting paths, so $C$ must also be of weight $0$ and hence so
	is $P'$.   We can then rule out possible cycles by considering
	the lowest-weight path from the initial state to the first state of $C$
	and similarly from the first state of $C$ to a final state; these
	weights must sum to $0$.
	Of the $26$ possible cycles, only three possibilities
	remain (up to choice of the initial state of the cycle).
	By considering the possible cycles and how they
	join up, we get the automaton in Figure~\ref{wt0}.
	
	\begin{figure}[h]
		\begin{center}
			\includegraphics[width=6.3in]{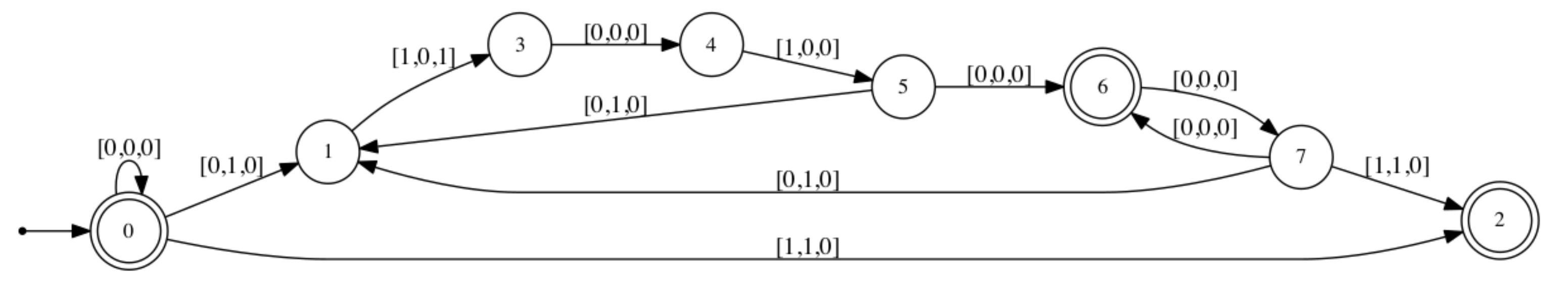}
		\end{center}
		\caption{Automaton for weight-$0$ paths.}
		\label{wt0}
	\end{figure}

	To get the accepted set of $n$, we project to the first coordinate
	and determinize the resulting automaton.  Thus we have proved
	the following result:
	\begin{figure}[h]
		\begin{center}
			\includegraphics[width=4.2in]{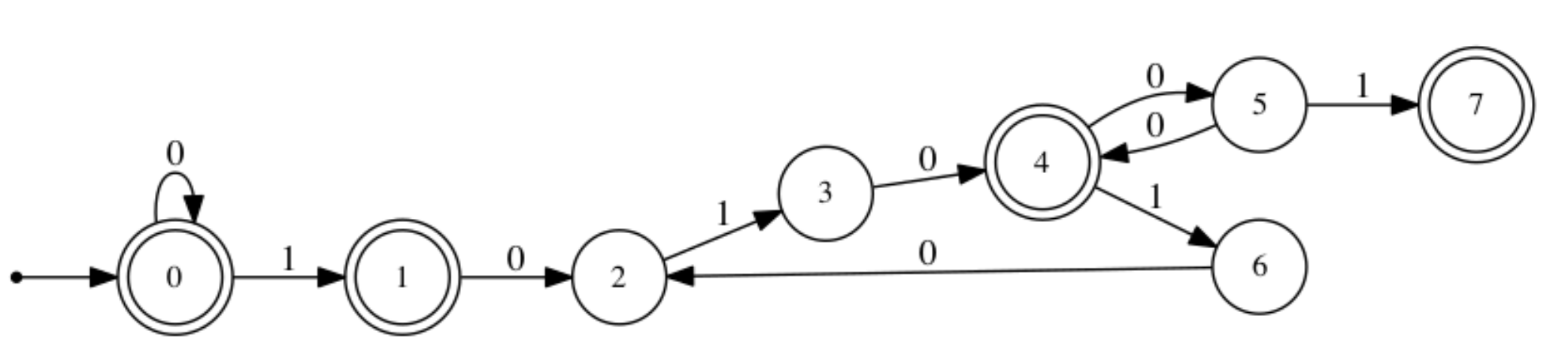}
		\end{center}
		\caption{Automaton for weight-$0$ paths.}
		\label{weight0}
	\end{figure}
	\begin{theorem}
		The automaton in Figure~\ref{weight0} accepts precisely those $n$ in Fibonacci representation
		for which $|(n)_F|_1 = |(n)_\varphi|_1$.
	\end{theorem}
	
	Similarly, one can study the sum of all the bits of the base-$\varphi$
	representation, taken modulo $2$.
	\begin{verbatim}
		reg sum2 {0,1} {0,1} "([0,0]|[1,1])*(([0,1]|[1,0])([0,0]|[1,1])*
		([0,1]|[1,0])([0,0]|[1,1])*)*([0,1]|[1,0])([0,0]|[1,1])*":
		# sum of all the bits of x and y, mod 2
		
		def sdpr2 "?msd_fib Ex,y $saka(n,x,y) & $sum2(x,y)":
		# sum of digits of phi-rep mod 2 is 1
	\end{verbatim}
	
	This gives a $73$-state automaton for this binary sequence.  It is
	sequence \seqnum{A330037} in the OEIS.

	The automata {\tt frougny} and {\tt saka} that we have constructed
	make it very easy to compute other quantities related to $\varphi$-representations.
	For example, let us define $s_L(n)$ to be the sum of the bits of the
	left part of the $\varphi$-representation of $n$.   This is  apparently
	a new sequence, never before studied; it is
	sequence \seqnum{A362716} in the
	OEIS.  (Dekking
	\cite{Dekking:2020,Dekking:2021} studied the sum $s(n)$ of {\it all\/} the bits
	of the $\varphi$-representation.)
	\begin{table}[htb]
		\begin{center}
			\begin{tabular}{c|ccccccccccccccccccccc}
				$n$ & 0 & 1 & 2 & 3 & 4 & 5 & 6 & 7 & 8 & 9 & 10 & 11 & 12 & 13 & 14 & 15 & 16 &17 & 18 & 19 \\
				\hline
				$s_L(n)$ & 0&1&1&1&2&1&2&1&2&2&2&3&1&2&2&3&2&3&1&2 
			\end{tabular}
		\end{center}
		\caption{First few values of $s_L(n)$.}
	\end{table}
	
	We can easily find a linear representation for $s_L(n)$ as follows:
	\begin{verbatim}
		reg match1 {0,1} {0,1} "([0,0]|[1,0])*[1,1]([0,0]|[1,0])*":
		def count1 "?msd_fib Ex,y $saka(n,x,y) & $match1(x,t)":
		def sl n "?msd_fib $count1(n,t)":
	\end{verbatim}
	This gives us a linear representation of rank $99$, which can
	be minimized to the following linear representation
	$(v, \gamma, w)$ of rank $19$:
	$$
	v = \left[ \begin{smallarray}{ccccccccccccccccccc}
		1&0&0&0&0&0&0&0&0&0&0&0&0&0&0&0&0&0&0
	\end{smallarray}\right] 
	$$
	\begin{align*}
		\gamma(0) & = \left[ \begin{smallarray}{ccccccccccccccccccc}
			1& 0& 0& 0& 0& 0& 0& 0& 0& 0& 0& 0& 0& 0& 0& 0& 0& 0& 0\\
			0& 0& 1& 0& 0& 0& 0& 0& 0& 0& 0& 0& 0& 0& 0& 0& 0& 0& 0\\
			0& 0& 0& 1& 0& 0& 0& 0& 0& 0& 0& 0& 0& 0& 0& 0& 0& 0& 0\\
			0& 0& 0& 0& 0& 1& 0& 0& 0& 0& 0& 0& 0& 0& 0& 0& 0& 0& 0\\
			0& 0& 0& 0& 0& 0& 0& 1& 0& 0& 0& 0& 0& 0& 0& 0& 0& 0& 0\\
			0& 0& 0& 0& 0& 0& 0& 0& 1& 0& 0& 0& 0& 0& 0& 0& 0& 0& 0\\
			0& 0& 0& 0& 0& 0& 0& 0& 0& 0& 1& 0& 0& 0& 0& 0& 0& 0& 0\\
			0& 0& 0& 0& 0& 0& 0& 0& 0& 0& 0& 1& 0& 0& 0& 0& 0& 0& 0\\
			0& 0& 1&-1& 0& 0& 0& 0& 0& 0& 1& 0& 0& 0& 0& 0& 0& 0& 0\\
			0& 0& 0& 0& 0& 0& 0& 0& 0& 0& 0& 0& 0& 0& 1& 0& 0& 0& 0\\
			0& 0& 0&-1& 0& 1& 0& 0& 0& 0& 1& 0& 0& 0& 0& 0& 0& 0& 0\\
			0& 0& 0& 0& 0& 0& 0& 0& 0& 0& 0& 0& 0& 0& 0& 0& 1& 0& 0\\
			0& 0& 1&-1& 0& 0& 0& 0& 0& 0& 1& 0& 0& 0& 0& 0& 0& 0& 0\\
			0& 0& 0&-1& 0& 1& 0& 0& 0& 0& 1& 0& 0& 0& 0& 0& 0& 0& 0\\
			0& 0& 0& 0& 0& 0& 0& 0& 0& 0& 0& 0& 0& 0& 0& 0& 0& 1& 0\\
			0& 0& 0&-1& 0& 0& 0& 0& 0& 0& 2& 0& 0& 0& 0& 0& 0& 0& 0\\
			0& 0& 0& 0& 0& 0& 0& 0& 0& 0& 0& 0& 0& 0& 0& 0& 0& 0& 1\\
			0& 0& 0&-1& 0& 0& 0& 0& 0& 0& 2& 0& 0& 0& 0& 0& 0& 0& 0\\
			0& 0& 0& 0& 0& 0& 0&-1& 0& 0& 0& 0& 0& 0& 0& 0& 2& 0& 0
		\end{smallarray} \right]
		&\!
		\gamma(1) &= \left[\begin{smallarray}{ccccccccccccccccccc}
			0& 1& 0& 0& 0& 0& 0& 0& 0& 0& 0& 0& 0& 0& 0& 0& 0& 0& 0\\
			0& 0& 0& 0& 0& 0& 0& 0& 0& 0& 0& 0& 0& 0& 0& 0& 0& 0& 0\\
			0& 0& 0& 0& 1& 0& 0& 0& 0& 0& 0& 0& 0& 0& 0& 0& 0& 0& 0\\
			0& 0& 0& 0& 0& 0& 1& 0& 0& 0& 0& 0& 0& 0& 0& 0& 0& 0& 0\\
			0& 0& 0& 0& 0& 0& 0& 0& 0& 0& 0& 0& 0& 0& 0& 0& 0& 0& 0\\
			0& 0& 0& 0& 0& 0& 0& 0& 0& 1& 0& 0& 0& 0& 0& 0& 0& 0& 0\\
			0& 0& 0& 0& 0& 0& 0& 0& 0& 0& 0& 0& 0& 0& 0& 0& 0& 0& 0\\
			0& 0& 0& 0& 0& 0& 0& 0& 0& 0& 0& 0& 1& 0& 0& 0& 0& 0& 0\\
			0& 0& 0& 0& 0& 0& 0& 0& 0& 0& 0& 0& 0& 1& 0& 0& 0& 0& 0\\
			0& 0& 0& 0& 0& 0& 0& 0& 0& 0& 0& 0& 0& 0& 0& 0& 0& 0& 0\\
			0& 0& 0& 0& 0& 0& 0& 0& 0& 0& 0& 0& 0& 0& 0& 1& 0& 0& 0\\
			0& 1& 0& 0& 0& 0&-1& 0& 0& 0& 0& 0& 0& 0& 0& 1& 0& 0& 0\\
			0& 0& 0& 0& 0& 0& 0& 0& 0& 0& 0& 0& 0& 0& 0& 0& 0& 0& 0\\
			0& 0& 0& 0& 0& 0& 0& 0& 0& 0& 0& 0& 0& 0& 0& 0& 0& 0& 0\\
			0& 0& 0& 0& 0& 0& 0& 0& 0& 0& 0& 0& 0& 1& 0& 0& 0& 0& 0\\
			0& 0& 0& 0& 0& 0& 0& 0& 0& 0& 0& 0& 0& 0& 0& 0& 0& 0& 0\\
			0& 0& 0& 0& 0& 0& 0& 0& 0& 0& 0& 0& 1& 0& 0& 0& 0& 0& 0\\
			0& 0& 0& 0& 1& 0&-1& 0& 0& 0& 0&-1& 0& 0& 0& 1& 0& 0& 1\\
			0& 1& 0& 0& 0& 0&-1& 0& 0& 0& 0& 0& 0& 0& 0& 1& 0& 0& 0
		\end{smallarray}\right]
		& \!
		w &= \left[ \begin{smallarray}{c}
			0\\
			1\\
			1\\
			1\\
			2\\
			1\\
			2\\
			1\\
			2\\
			2\\
			2\\
			3\\
			1\\
			2\\
			3\\
			3\\
			1\\
			3\\
			4
		\end{smallarray}
		\right] .
	\end{align*}
	
	Similarly, one can study the sum-of-digits function
	$s_R (n)$ for the right part of the
	canonical $\varphi$-representation.   We prove a new theorem about it:
	\begin{theorem}
		The difference $d(n) := s_R(n+1)-s_R(n)$ is a Fibonacci-automatic sequence
		taking values in $\{ -1,0,1\}$ only.
	\end{theorem}
	
	\begin{proof}
		We use the following {\tt Walnut} code.
		\begin{verbatim}
			reg match1 {0,1} {0,1} "([0,0]|[1,0])*[1,1]([0,0]|[1,0])*":
			def countr "?msd_fib Ex,y $saka(n,x,y) & $match1(y,u)":
			def sr n "?msd_fib $countr(n,t)":
			def sr1 n "?msd_fib $countr(n+1,t)":
		\end{verbatim}
		This produces linear representations for
		$s_R(n)$ (rank $82$; minimizes to rank $21$)
		and $s_R(n+1)$ (rank $64$; minimizes to rank $17$).
		From this we can produce a linear representation
		for $d(n) = s_R(n+1)-s_R(n)$ of rank $38$, which minimizes
		to one of rank $17$.  We can then do the ``semigroup trick''
		\cite[\S 4.11]{Shallit:2023}
		to show that the resulting sequence is Fibonacci automatic.
		The automaton is depicted in Figure~\ref{srd}.
		\begin{figure}[htb]
			\begin{center}
				\includegraphics[width=6.3in]{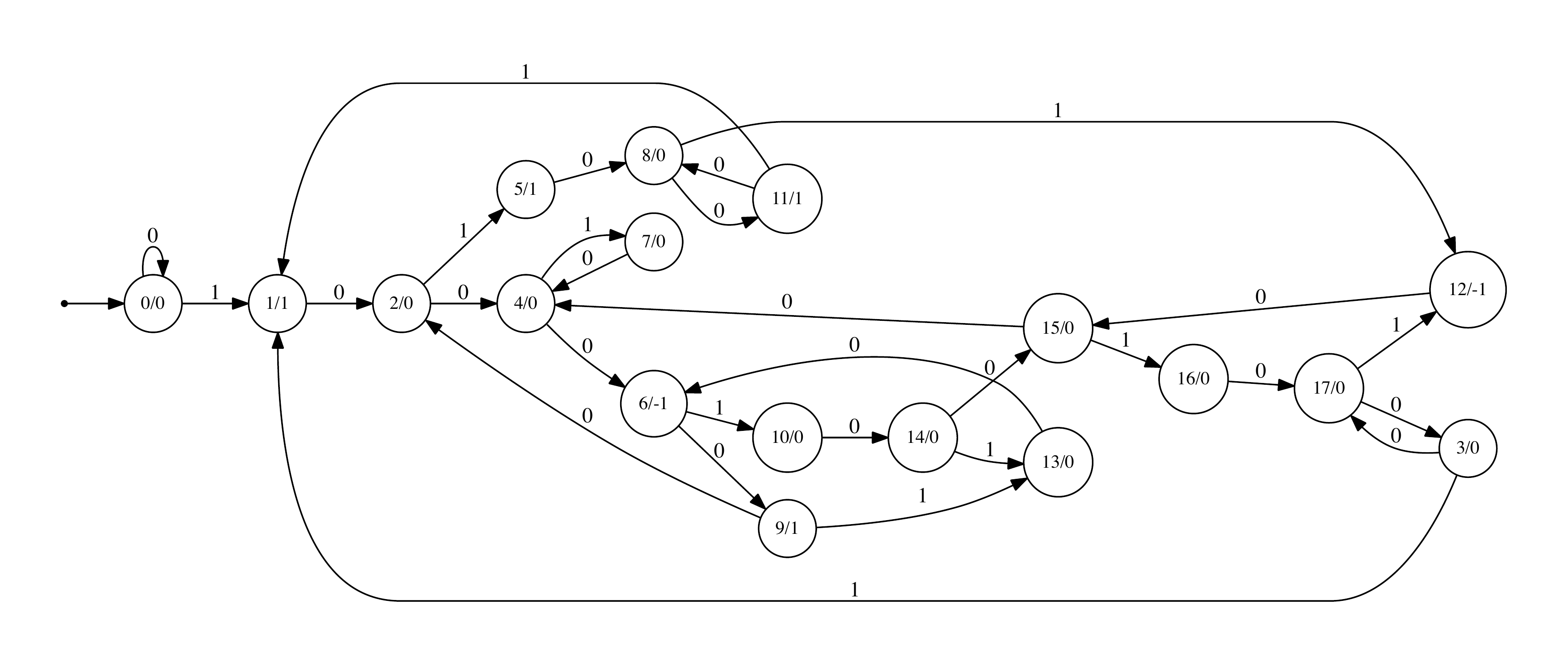}
			\end{center}
			\caption{Fibonacci automaton for $d(n)$.}
			\label{srd}
		\end{figure}
	\end{proof}

	\section{Length of base-\texorpdfstring{$\varphi$}{φ} representations}
	\label{fivep}
	
	Similarly, one can study the length $\ell(n)$
	of the left part of $(n)_\varphi$ (or, alternatively, one more than the
	exponent of the highest power of $\varphi$ appearing in $(n)_\varphi$).

	This is sequence \seqnum{A362692} in the OEIS.

	Here the idea is to create a simple automaton, {\tt onepos},
	that accepts binary strings
	$x$ and $t$ if and only if $t$ contains a single $1$ and it occurs
	at or to the right of the first $1$ in $x$.  Then we simply count the
	number of such $t$:
	\begin{verbatim}
		def intlength "?msd_fib Ex,y $saka(n,x,y) & $onepos(x,t)":
		def intl n "?msd_fib $intlength(n,t)":
	\end{verbatim}
	The resulting linear representation can be minimized to the
	following linear representation of rank $9$:
	\begin{align*}
		v^T &= \left[ \begin{smallarray}{ccccccccc}
			1\\
			0\\
			0\\
			0\\
			0\\
			0\\
			0\\
			0\\
			0
		\end{smallarray} \right] 
		&
		\gamma(0) & = \left[ \begin{smallarray}{ccccccccc}
			1& 0& 0& 0& 0& 0& 0& 0& 0\\
			0& 0& 1& 0& 0& 0& 0& 0& 0\\
			0& 0& 0& 1& 0& 0& 0& 0& 0\\
			0& 0& 0& 0& 0& 1& 0& 0& 0\\
			0& 0& 0& 0& 0& 0& 0& 1& 0\\
			0& 0& 0&-1& 0& 2& 0& 0& 0\\
			0& 0& 0&-1& 0& 2& 0& 0& 0\\
			0& 0& 0&-2& 1& 3&-1& 0& 0\\
			0& 0& 0&-2& 0& 3& 0& 0& 0
		\end{smallarray} \right]
		&
		\gamma(1) &= \left[\begin{smallarray}{ccccccccc} 
			0& 1& 0& 0& 0& 0& 0& 0& 0\\
			0& 0& 0& 0& 0& 0& 0& 0& 0\\
			0& 0& 0& 0& 1& 0& 0& 0& 0\\
			0& 0& 0& 0& 0& 0& 1& 0& 0\\
			0& 0& 0& 0& 0& 0& 0& 0& 0\\
			0& 0& 0& 0& 0& 0& 0& 0& 1\\
			0& 0& 0& 0& 0& 0& 0& 0& 0\\
			0& 0& 0& 0& 0& 0&-1& 0& 2\\
			0& 0& 0& 0& 0& 0& 0& 0& 0
		\end{smallarray}\right]
		&
		w &= \left[ \begin{smallarray}{c}
			0\\
			1\\
			2\\
			3\\
			3\\
			4\\
			4\\
			5\\
			5
		\end{smallarray} 
		\right] .
	\end{align*}
	
	\begin{table}[h]
		\begin{center}
			\begin{tabular}{c|ccccccccccccccccccccc}
				$n$ & 0 & 1 & 2 & 3 & 4 & 5 & 6 & 7 & 8 & 9 & 10 & 11 & 12 & 13 & 14 & 15 & 16 &17 & 18 & 19 \\
				\hline
				$\ell(n)$ & 0&1&2&3&3&4&4&5&5&5&5&5&6&6&6&6&6&6&7&7 
			\end{tabular}
		\end{center}
		\caption{First few values of $\ell(n)$.}
	\end{table}
	
	Inspection suggests that $\ell(n)$ is an increasing sequence and
	the gaps between successive terms are only $0$ and $1$.  
	
	We can prove
	this by finding the linear representation for $\ell(n)-\ell(n-1)$ and
	then using the ``semigroup trick''.
	This gives a Fibonacci DFAO, depicted
	in Figure~\ref{fig4}, computing
	$\ell(n) - \ell(n-1)$.
	\begin{figure}[htb]
		\begin{center}
			\includegraphics[width=6in]{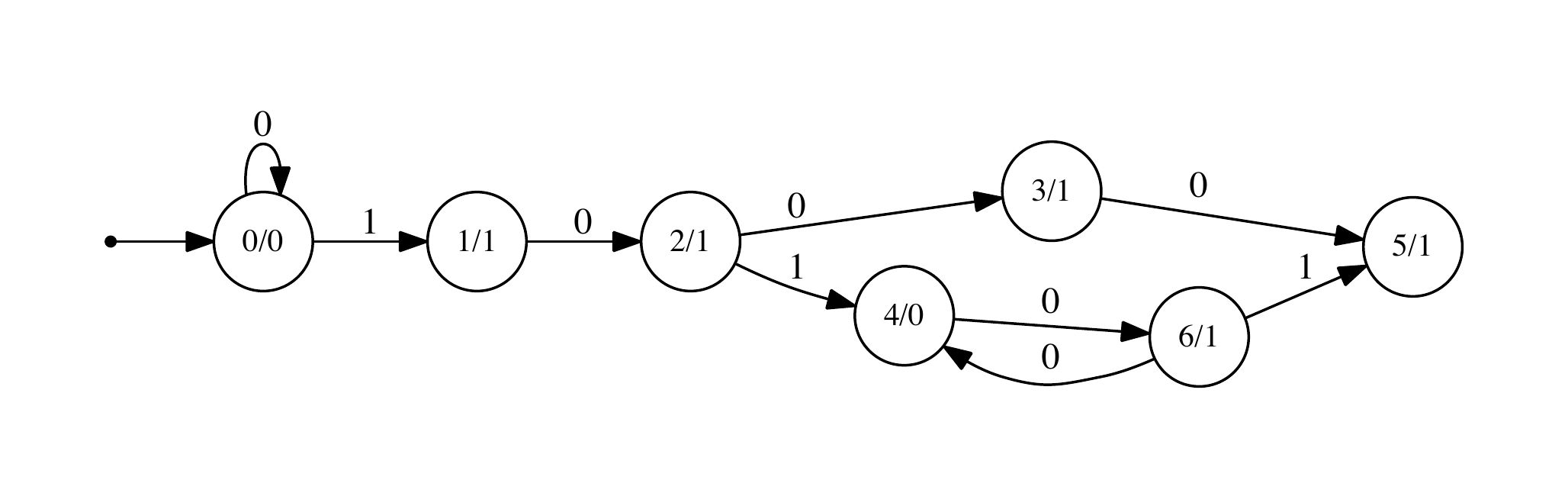}
		\end{center}
		\caption{Fibonacci DFAO for first difference of $\ell(n)$.}
		\label{fig4}
	\end{figure}
	
	Inspection of this automaton and observing
	that the Fibonacci representation of the Lucas numbers
	are of the form $1010*$, gives the following result:
	\begin{theorem}
		We have $\ell(n)-\ell(n-1) = 1$ if and only
		if $n=1$ or $n = L_{2i}$ or $n = L_{2i-1} + 1$ for $i \geq 1$.
	\end{theorem}
	This is implied by a result of Sanchis and Sanchis
	\cite[Thm.~2.1]{Sanchis&Sanchis:2001}.
	
	\section{Individual bits}
	\label{six}
	
	We can also prove theorems about individual bits
	of the $\varphi$-representation.
	Let 
	$$n = \sum_{L \leq i \leq R} d_i(n) \varphi^i$$ be the canonical 
	$\varphi$-representation
	of the natural number $n$ and let
	\[\alpha = (3-\varphi)/5 = (5-\sqrt{5})/10.\]
	The Sturmian sequence based on 
	$\alpha$ is defined to be $w(n) = \lfloor (n+1) \alpha \rfloor - 
	\lfloor n \alpha \rfloor$; this is sequence
	\seqnum{A221150} in the OEIS. For more about Sturmian sequences,
	see, e.g., \cite[Chap.~2]{Lothaire:2002}.
	
	The following result characterizes $d_0(n)$ and $d_1(n)$.
	Some previous papers have characterized these two sequences
	(see \cite{Hart&Sanchis:1999,Dekking:2020b}),
	but it appears that the specific characterization
	we present here is new.
	\begin{theorem}
		We have
		\begin{itemize}
			\item[(a)] $d_0(n+1) = w(n)$ for $n \geq 1$;
			\item[(b)] $d_1(n) = w(n+1)$ for $n \geq 0$.
		\end{itemize}
	\end{theorem}
	
	\begin{proof}
		\leavevmode
		\begin{itemize}
			\item[(a)]  We can compute $w(n)$ with {\tt Walnut} using the
			existing code {\tt phin}, which computes $\lfloor \varphi n \rfloor$ (see
			\cite[p.~278]{Shallit:2023}).  The resulting automaton computes $w$ in a
			``synchronized'' fashion; that is,  the inputs are $n$ and $x$, both
			in Zeckendorf representation, and the automaton accepts
			if and only if $x = w(n)$.
			Now
			$$
			\lfloor \alpha n \rfloor  = \left\lfloor \frac{3-\varphi}{5} n \right\rfloor 
			= \left\lfloor \frac{\lfloor (3-\varphi)n \rfloor}{5} \right\rfloor 
			= \left\lfloor \frac{3n - 1 - \lfloor \varphi n \rfloor}{5} \right\rfloor 
			$$
			for $n \geq 1$.  It follows that a synchronized
			automaton for $w(n)$ can be computed by the
			following {\tt Walnut} code. 
			\begin{verbatim}
				reg shift {0,1} {0,1} "([0,0]|[0,1][1,1]*[1,0])*":
				def phin "?msd_fib (s=0 & n=0) | Ex $shift(n-1,x) & s=x+1":
				def alphan "?msd_fib Ex $phin(n,x) & z=(3*n-(x+1))/5":
				def w "?msd_fib Ez $alphan(n+1,z+1) & $alphan(n,z)":
			\end{verbatim}
			Similarly, we can construct a synchronized automaton for $d_0$ and
			verify claim (a) as follows:
			\begin{verbatim}
				def d0 "?msd_fib Ex,y $saka(n,x,y) & $lstbit1(x,1)":
				eval checka "?msd_fib An (n>=1) => ($w(n) <=> $d0(n+1))":
			\end{verbatim}
			and {\tt Walnut} returns {\tt TRUE}.
			\item[(b)] We use the following {\tt Walnut} code, to compute
			a synchronized automaton for $d_1$ and check the assertion:
			\begin{verbatim}
				def d1 "?msd_fib Ex,y $saka(n,x,y) & $lstbit2(x,1)":
				eval checkb "?msd_fib An ($w(n+1) <=> $d1(n))":
			\end{verbatim}
			and {\tt Walnut} returns {\tt TRUE}.
		\end{itemize}
	\end{proof}
	
	Now let us turn to $d_{-1} (n)$.   Of course, this is a Fibonacci-automatic
	sequence (and can be computed by a DFAO of $13$ states).  We can
	find a surprising connection for $d_{-1}(n)$ by consulting the OEIS,
	as follows.
	
	Instead of representations involving sums of Fibonacci numbers,
	we can consider sums of Lucas numbers $L_n$, defined
	previously.  The following result
	was proved by Brown \cite{Brown:1969}:
	every natural number can be written as a sum
	$\sum_{0 \leq i \leq t} e_i L_i$, where $e_i \in \{0,1\}$.
	We use the shorthand notation $[e_t \cdots e_1 e_0]_L$ for this sum.
	However, this representation is not always unique, even if
	we impose the natural condition that 
	\begin{equation}
		e_i e_{i+1} \not=1 \ \text{for all $i$}.
		\label{natcon}
	\end{equation}
	For example, we have $[1010]_L = L_3 + L_1 = 5 = L_2 + L_0 = [101]_L$.   
	Recently, Chu, Luo, and Miller \cite{Chu&Luo&Miller:2022}
	proved that if we impose 
	Condition~\eqref{natcon}, then there are at most
	two distinct Lucas representations for each $n$.
	The set of $n$ with two distinct
	Lucas representations forms sequence \seqnum{A342089} in the OEIS.
	
	\begin{theorem}
		For all $n \geq 0$ we have $d_{-1} (n) = 1$ if and only
		if $n$ has exactly two distinct Lucas representations.
	\end{theorem}
	
	\begin{proof}
		We start by constructing an automaton to evaluate a Lucas
		representation.  It takes a number $n$
		in Zeckendorf representation and a
		binary string $x$ as inputs and accepts 
		if $[x]_L = n$.
		
		To do so, we use the classic identity $L_n = F_{n+1} + F_{n-1}$ for 
		all $n$.  It follows that $[xabc]_L = [xab]_F + [x]_F + a+2c $ for 
		for binary strings $x$ and $a,b,c \in \{0,1\}$.  We can therefore
		compute it as follows:
		\begin{verbatim}
			def luc2fib "?msd_fib Eu,v,w,a,c,y,z $shiftr(x,u) & $shiftr(u,v) &
			$shiftr(v,w) & $lstbit3(x,a) & $lstbit1(x,c) & $fibnorm(u,y) &
			$fibnorm(w,z) & n=y+z+a+2*c":
		\end{verbatim}
		Next, we create an automaton to accept the Fibonacci representation
		of those $n$ with two different Lucas representations.
		\begin{verbatim}
			reg has11 {0,1} "(0|1)*11(0|1)*":
			reg equal {0,1} {0,1} "([0,0]|[1,1])*":
			def tworep "?msd_fib Ex,y (~$has11(x)) & (~$has11(y)) &
			(~$equal(x,y)) & $luc2fib(n,x) & $luc2fib(n,y)":
		\end{verbatim}
		Finally, we create a synchronized automaton
		for $d_{-1}(n)$ and check its equality
		with sequence \seqnum{A342089}:
		\begin{verbatim}
			def dm1 "?msd_fib Ex,y $saka(n,x,y) & $lstbit1(y,1)":
			eval checkseq "?msd_fib An ($dm1(n) <=> $tworep(n))":
		\end{verbatim}
		and {\tt Walnut} returns {\tt TRUE}.
	\end{proof}
	
	\begin{remark}
		Hart and Sanchis \cite{Hart&Sanchis:1999} proved that the limiting
		frequency of $1$'s in $(d_{-1} (n))_{n \geq 0}$, namely
		$\lim_{n \rightarrow \infty} {\frac{1}{n}} \sum_{0 \leq i < n} d_{-1} (n)$,
		is $\gamma := 1/(3\varphi+1) \doteq 0.17082$.   We can 
		easily prove this with {\tt Walnut}
		as follows:  first, define $d'(n) = \sum_{0 \leq i \leq n} d_{-1} (i)$.
		Then we can ``guess'' a synchronized automaton
		for $d'(n)$ using the method outlined in \cite[\S 10.15]{Shallit:2023}.
		Next, we verify that our guess is correct using the automaton
		{\tt dm1} computed above.   Finally, we show
		with {\tt Walnut} that $|d'(n) - w'(n)| \leq 1$, where
		$w'(n)$ is the Sturmian sequence corresponding to $\gamma$.
		We omit the details.
	\end{remark}
	
	We now turn to considering what Dekking and van Loon \cite{Dekking&van.Loon:2021}
	called ``vertical runs of $1$'s''.   Given $i \in \Zee$, we saw above
	that $d_i (n)$ is the coefficient of $\alpha^i$ in the canonical
	$\varphi$-representation of $n$.   A {\it vertical run of $1$'s\/} 
	takes place starting at position $n$ if
	$$ d_i(n-1) = 0, \quad d_i(n) = 1, \quad d_i(n+1)=1, \ldots, d_i(n+t-1) = 1, \quad d_i (n+t) = 0.$$
	and we say the length of the run is $t$.   (If $n = 0$ we assume $d_i(-1) = 0$
	for all $i$.)
	The name comes from thinking of the base-$\varphi$ representation of
	the numbers $d_i(n-1), \ldots, d_i(n+t)$ written vertically, with the
	$i$'th position lined up.  For example, for $i = 4$ and $n= 7$ we have
	a vertical run of length $5$, as illustrated in Table~\ref{tab7}.
	\begin{table}[htb]
		\begin{center}
			\begin{tabular}{c|l}
				$n$ & $(n)_\varphi$ \\
				\hline
				6 & $0{\bf0}1010.0001$ \\
				7 & $0{\bf 1}0000.0001$ \\
				8 & $0{\bf 1}10001.0001$ \\
				9 & $0{\bf 1}0010.0101$ \\
				10 & $0{\bf 1}0100.0101$ \\
				11 & $0{\bf 1}0101.0101$ \\
				12 & $1{\bf 0}0000.101001$
			\end{tabular}
		\end{center}
		\caption{A vertical run of $5$ $1$'s.}
		\label{tab7}
	\end{table}
	
	For each $i$ we can determine the possible lengths of vertical runs.
	\begin{theorem}
		For the canonical $\varphi$-representation
		the vertical runs of $1$'s have length $v(i)$, as follows:
		$$ v(i) = \begin{cases}
			L_{i-1} + (-1)^i, & \text{if $i \geq 1$}; \\
			1, & \text{if $i=0$}; \\
			L_{-i}, & \text{if $i< 0$}.
		\end{cases}
		$$
		\label{vertthm}
	\end{theorem}
	\begin{proof}
		We can create a {\tt Walnut} formula that determines the possible
		lengths of vertical runs for each $i \geq 0$, as follows:
		\begin{verbatim}
			def matchfib "?msd_fib Ex,y $saka(n,x,y) & $isfib(t) & $match1(x,t)":
			def verticalr "?msd_fib (i>0) & En (~$matchfib(n-1,t)) &
			(~$matchfib(n+i,t)) & Aj (j<i) => $matchfib(n+j,t)":
		\end{verbatim}
		The resulting automaton accepts the Zeckendorf representation of $v(i)$
		in parallel with a binary string with exactly one $1$, occurring
		$i$ bits from the right end.  It is displayed in Figure~\ref{vert}.
		\begin{figure}[htb]
			\begin{center}
				\includegraphics[width=6in]{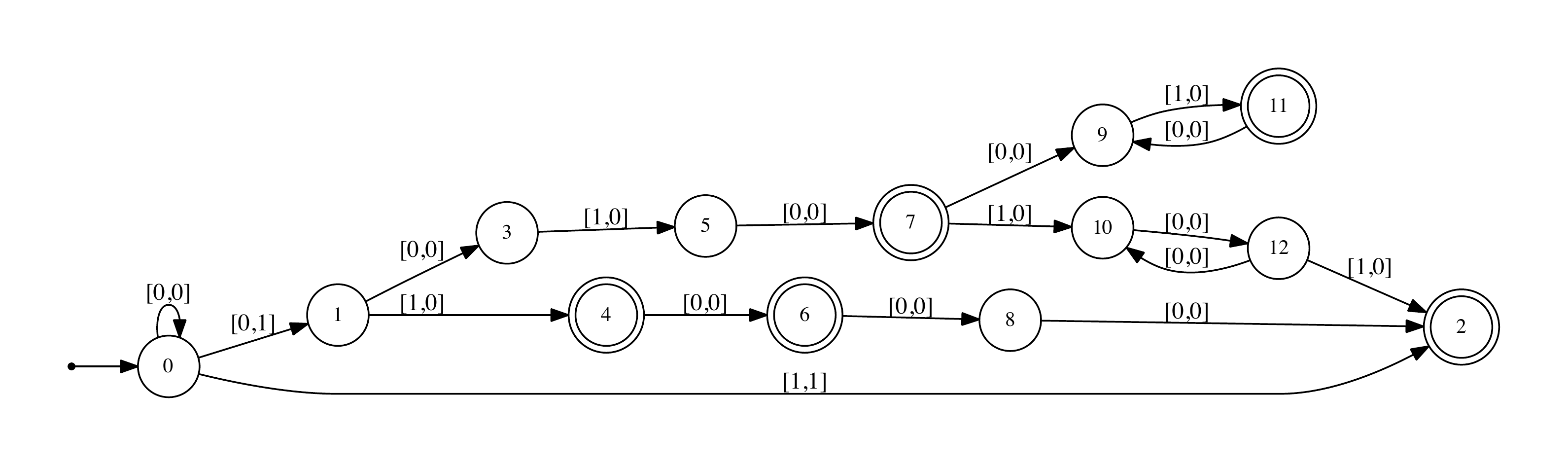}
			\end{center}
			\caption{Automaton for $v(i)$, $i> 0$.}
			\label{vert}
		\end{figure}
		
		As one can see by inspecting the automaton in Figure~\ref{vert}, the
		acceptance paths, ignoring leading $0$'s, are of the following types:
		\begin{itemize}
			\item[(a)] $[1,1]$ 
			\item[(b)] $[0,1][1,0]$
			\item[(c)] $[0,1][1,0][0,0]$
			\item[(d)] $[0,1][0,0][1,0][0,0]$
			\item[(e)] $[0,1][1,0][0,0][0,0][0,0]$
			\item[(f)] $[0,1][0,0][1,0][0,0]([0,0][1,0])^*$ 
			\item[(g)] $[0,1][0,0][1,0][0,0][1,0]([0,0][0,0])^*[0,0][1,0]$
		\end{itemize}
		These correspond to the following accepted inputs:
		\begin{itemize}
			\item[(a)] $i = 0$, $v(0) = 1$;
			\item[(b)] $i = 1$, $v(1) = 1$;
			\item[(c)] $i = 2$, $v(2) = 2$;
			\item[(d)] $i = 3$, $v(3) = 2$;
			\item[(e)] $i = 4$, $v(5) = 5$;
			\item[(f)] $i=2k+1$ odd, $v(i) = L_{i-1} - 1$;
			\item[(g)] $i = 2k$ even, $v(i) = L_{i-1} + 1$.
		\end{itemize}
		The case of $i<0$ can be handled similarly; we omit the details.
	\end{proof}
	
	\begin{remark}
		Dekking and Van Loon \cite{Dekking&van.Loon:2021}
		claimed that ``there is no such
		regularity'', but Theorem~\ref{vertthm} would seem to contradict that.
	\end{remark}

	\section{Palindromic and anti-palindromic \texorpdfstring{$\varphi$}{φ}-representations}
	\label{seven}
	
	Let us consider natural numbers with palindromic $\varphi$-representations,
	that is, where $n = [x.x^R]_\varphi$.   There are two variations:  one
	where we demand that the expansion be canonical, and one where
	we do not make this assumption.
	
	Let us start with the canonical case.   Since the representation that
	our automaton {\tt saka} uses is ``folded'', we can find the Zeckendorf
	representation of those $n$ with palindromic $\varphi$-representations
	with the following {\tt Walnut} code:
	\begin{verbatim}
		def palcanon "?msd_fib Ex,y $saka(n,x,y) & $equal(x,y)":
	\end{verbatim}
	\begin{figure}[H]
		\begin{center}
			\includegraphics[width=6.2in]{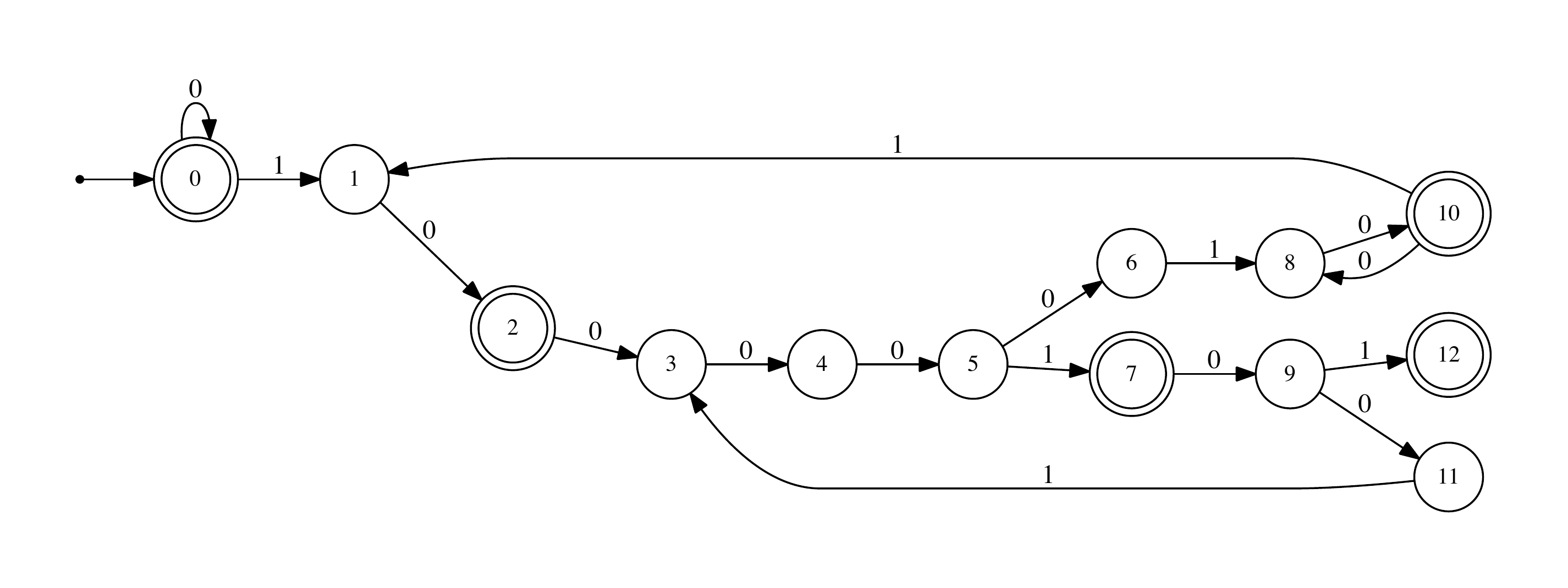}
		\end{center}
		\caption{Fibonacci automaton for those $n$ with palindromic $\varphi$-representations.}
		\label{palcanon}
	\end{figure}
	Hence we have shown the following
	\begin{theorem}
		The natural number 
		$n$ has a palindromic canonical $\varphi$-representation if and only
		if the $12$-state automaton depicted in Figure~\ref{palcanon}
		accepts $(n)_F$.
	\end{theorem}
	The resulting sequence of $n$ accepted by this automaton
	is 
	$$2, 14, 36, 38, 94, 96, 246, 248, 260,\ldots$$ and forms sequence
	\seqnum{A362780} in the OEIS.
	
	Now we turn to the case of allowing non-canonical expansions. 
	Here there are additional examples such as
	$6 = [1001.1001]$, which is palindromic,
	non-canonical due to the presence of $11$.
	We can construct an automaton for these $n$ as follows:
	\begin{verbatim}
		def pal "?msd_fib Ex,y $frougny(n,x,y) & $equal(x,y)":
	\end{verbatim}
	This proves the following result:
	\begin{theorem}
		The natural number $n$ has some palindromic (possibly non-canonical)
		$\varphi$-representation if and only
		if the $16$-state automaton depicted in Figure~\ref{pal}
		accepts $(n)_F$.
		\begin{figure}[htb]
			\begin{center}
				\includegraphics[width=6.2in]{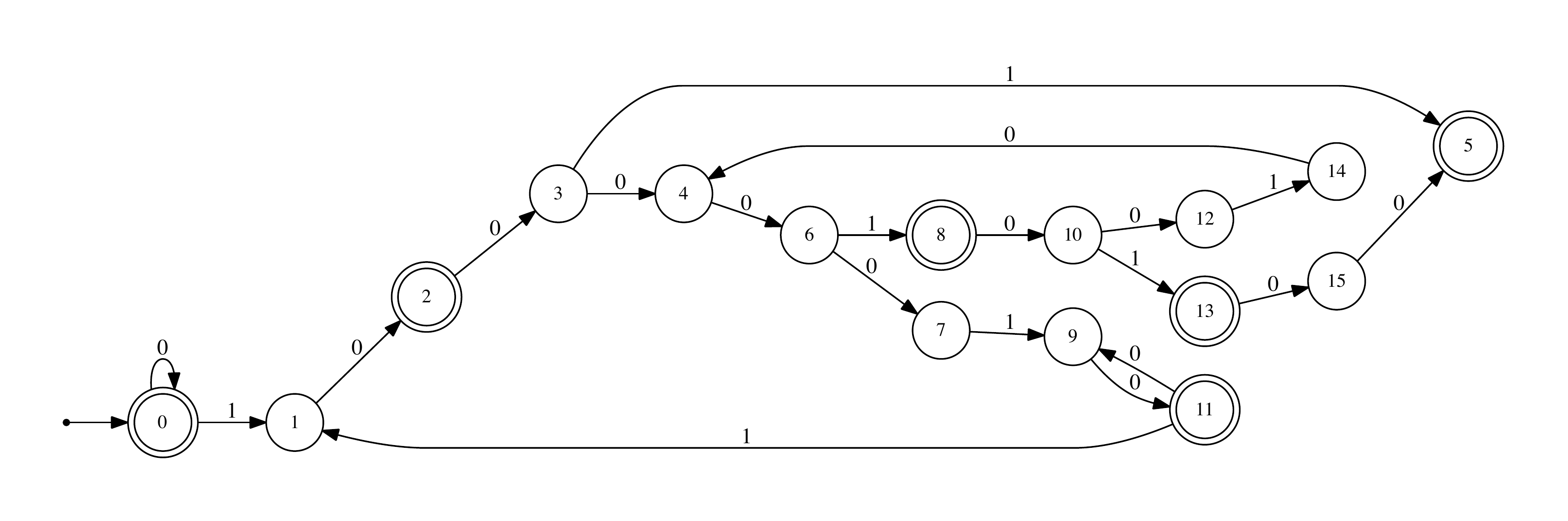}
			\end{center}
			\caption{Fibonacci automaton for those $n$ having some palindromic (possibly 
				non-canonical) $\varphi$-representations.}
			\label{pal}
		\end{figure}
	\end{theorem}
	The resulting sequence of $n$ accepted by this automaton
	is $$2, 6, 14, 36, 38, 94, 96, 100, 246, 248, 252, 260,\ldots$$
	and forms sequence
	\seqnum{A330672} in the OEIS.  We remark that the only new examples
	are those where the only $11$ occurs as the $1$ at the end of $x$
	(and hence at the beginning of $x^R$), as can be verified with the
	following {\tt Walnut} code:
	\begin{verbatim}
		def pal11 "?msd_fib Ex,y $frougny(n,x,y) & $equal(x,y) & $has11(x)":
	\end{verbatim}
	which accepts nothing.
	
	Next we turn to Shevelev's so-called ``$\varphi$-antipalindromic numbers";
	these are the $n$ for which the canonical base-$\varphi$ representation
	of $n$ is of the form $xa.x^R$, where $a \in \{0,1\}$.  (The name 
	comes from \seqnum{A178482} and is
	rather confusing, but we are keeping it.)
	\begin{verbatim}
		def shevanti "?msd_fib Ex,y $shiftr(x,y) & $saka(n,x,y)":
	\end{verbatim}
	\begin{theorem}
		Shevelev's $\varphi$-antipalindromic numbers $n$ are precisely those
		for which $(n)_F$ is accepted by the automaton in Figure~\ref{shev}.
		\begin{figure}[htb]
			\begin{center}
				\includegraphics[width=6.2in]{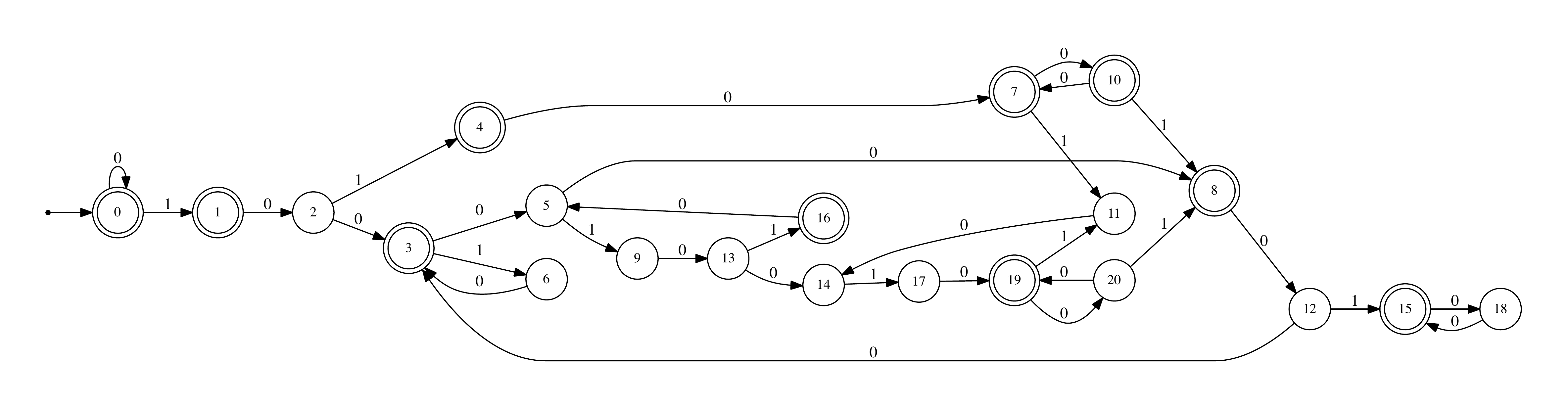}
			\end{center}
			\caption{Fibonacci automaton accepting Shevelev's $\varphi$-antipalindromic numbers.}
			\label{shev}
		\end{figure}
	\end{theorem}
	The resulting sequence of accepted $n$,
	$$ 0, 1, 3, 4, 7, 8, 10, 11, 18, 19, 21, 22, 25, 26, 28, 29, 47,\ldots$$
	forms sequence \seqnum{A178482} in the OEIS.
	
	Finally, we consider antipalindromic expansions, i.e.,
	expansions of the form $n = [x.\overline{x^R}]$, where the
	overline denotes a bitwise complement $0 \rightarrow 1, 1 \rightarrow 0$.
	Here we have to allow
	leading zeros in the left part and trailing zeros in the right part.
	\begin{theorem}
		There is a 193-state automaton that accepts precisely those $(n)_F$
		for which $n$ has a (possibly non-canonical) antipalindromic base-$\varphi$ 
		expansion.
	\end{theorem}
	\begin{proof}
		We use the following {\tt Walnut} code to produce the automaton:
		\begin{verbatim}
			reg compl {0,1} {0,1} "[0,0]*([0,1]|[1,0])*":
			def antip "?msd_fib $frougny(n,x,y) & $compl(x,y)":
		\end{verbatim}
	\end{proof}
	The first few terms of this sequence are
	$$ 1,3,4,5,6,8,11,13,14,15,16,21,23,29,31,33,35,37,39,41,43,45,\ldots . $$
	This is apparently a new sequence, never studied before, and is
	sequence~\seqnum{A362781} in the OEIS.
	
	Table~\ref{tab4} gives
	the first few antipalindromic expansions for these numbers.
	\begin{table}[htb]
		\begin{center}
			\begin{tabular}{c|l}
				$n$ & $[x.y]$ \\
				\hline
				1 & 1.0 \\
				3  &0011.0011\\
				4  &101.010\\
				5  &0110.1001\\
				6  &001001.011011\\
				8  &001100.110011\\
				11  &10101.01010\\
				13  &00100001.01111011\\
				14  &011011.001001\\
				15  &00100100.11011011
			\end{tabular}
		\end{center}
		\label{tab4}
		\caption{Antipalindromic expansions.}
	\end{table}
	
	\section{Canonical expansions with fixed number of \texorpdfstring{$1$}{1}'s}
	\label{eight}
	
	As further examples of what can be done, let us consider those
	$n$ for which $(n)_\varphi$ contains a given fixed number $t$ of $1$'s,
	for $2 \leq t \leq 5$.   We use the following code:
	\begin{verbatim}
		reg haszero1 {0,1}  "0*":
		reg hasone1 {0,1}   "0*10*":
		reg hastwo1 {0,1}   "0*10*10*":
		reg hasthree1 {0,1} "0*10*10*10*":
		reg hasfour1 {0,1}  "0*10*10*10*10*":
		reg hasfive1 {0,1}  "0*10*10*10*10*10*":
		def canon2 "?msd_fib Ex,y $saka(n,x,y) & (($haszero1(x)&$hastwo1(y))|
		($hasone1(x)&$hasone1(y))|($hastwo1(x)&$haszero1(y)))":
		def canon3 "?msd_fib Ex,y $saka(n,x,y) & (($haszero1(x)&$hasthree1(y))|
		($hasone1(x)&$hastwo1(y))|($hastwo1(x)&$hasone1(y))|
		($hasthree1(x)&$haszero1(y)))":
		def canon4 "?msd_fib Ex,y $saka(n,x,y) & (($haszero1(x)&$hasfour1(y))|
		($hasone1(x)&$hasthree1(y))|($hastwo1(x)&$hastwo1(y))|
		($hasthree1(x)&$hasone1(y))|($hasfour1(x)&$haszero1(y)))":
		def canon5 "?msd_fib Ex,y $saka(n,x,y) & (($haszero1(x)&$hasfive1(y))|
		($hasone1(x)&$hasfour1(y))|($hastwo1(x)&$hasthree1(y))|
		($hasthree1(x)&$hastwo1(y))|($hasfour1(x)&$hasone1(y))|
		($hasfive1(x)&$haszero1(y)))":
	\end{verbatim}
	We summarize our results below.
	\begin{theorem}
		For all $t \geq 0$ there is an automaton recognizing those
		$(n)_F$ such that $(n)_\varphi$ contains exactly $t$ $1$'s.
		For $2 \leq t \leq 5$ the number of states is given in Table~\ref{tab8}.
		\begin{table}[htb]
			\begin{center}
				\begin{tabular}{c|c|c}
					&  & sequence in OEIS  \\
					$t$ & number of states & enumerating $(n)_\varphi$ \\
					&                  & with $t$ $1$'s \\
					\hline 
					2 & 6 & \seqnum{A005248} \\
					3 & 9 & \seqnum{A104626} \\
					4 & 24 & \seqnum{A104627} \\
					5 & 46 & \seqnum{A104628}
				\end{tabular}
			\end{center}
			\caption{Number of states for automaton recognizing $(n)_F$ such that $(n)_\varphi$ contains exactly $t$ $1$'s.}
			\label{tab8}
		\end{table}
	\end{theorem}
	
	\section{Knott expansions}
	\label{three}
	
	Recently Dekking and Van Loon \cite{Dekking&van.Loon:2023}
	studied Knott representations of the positive integers, which
	are those $\varphi$-representation of $n$ that do not end
	in $011$ (followed, perhaps, by an arbitrary number of $0$'s).
	For example, $1.11$ and $10.01$ are Knott representations
	of $2$, but $1.1011$ and $10.0011$ are not Knott.
	They gave a method to compute $\totk(n)$, 
	the number of Knott representations
	of $n$ in terms of the Zeckendorf representation of $n$,
	but finding it and expressing it
	is rather complicated.   Here we show that,
	starting with the first Frougny-Sakarovitch automaton, one can
	easily find, with {\tt Walnut},
	a linear representation for $\totk(n)$ that permits
	efficient computation.
	All we have to do is find an automaton to check the Knott condition and
	ask {\tt Walnut} to compute the appropriate matrices,
	as follows:
	\begin{verbatim}
		def dekking "?msd_fib $frougny(n,x,y) & $knott(x,y)":
		def dek n "?msd_fib $dekking(n,x,y)":
	\end{verbatim}
	These give a linear representation
	$(v,\gamma,w)$ for $\totk(n)$ of rank $122$.
	This is sequence \seqnum{A289749} in the OEIS.
	
	This linear representation gives us a $O(\log n)$ method
	to compute $\totk(n)$, and much more.  It also allows
	us to compute closed forms for various kinds of $n$.
	For example, we can easily recover the following two results
	of \cite{Dekking&van.Loon:2023}:
	\begin{theorem}
		We have
		\begin{itemize}
			\item[(a)] $\totk(F_k) = F_k$ for $k \geq 1$;
			\item[(b)] $\totk(L_k) =
			\begin{cases}
				k,& \text{if $k$ odd}; \\
				k+1,& \text{if $k$ even}.
			\end{cases}$
		\end{itemize}
		\label{totk}
	\end{theorem}
	
	\begin{proof}
		The basic idea has already been explored in detail
		in a number of works; see \cite{Shallit:2023} for example.
		The basic idea is that if the Zeckendorf representation
		of $n$ looks like $r s^k t$ for strings $r, s, t$, then
		the linear representation for $f$ evaluates to
		$v \gamma(r) \gamma(s)^k \gamma(t) w$, and therefore is
		dependent on the entries of the $k$th power
		of the matrix $\gamma(s)$.   This, in turn, is governed
		by the zeros of the minimal polynomial of $\gamma(s)$.
		
		For our Theorem, we have $(F_k)_F = 10^{k-2}$ and
		$(L_k)_F = 1010^{k-3}$.  Therefore, the subsequences
		$\totk(F_k)$ and $\totk(L_k)$ can be expressed as a
		linear combination of the powers of the zeros of the
		minimal polynomial of $\gamma(0)$.
		
		We can ask {\tt Maple} (or any symbolic algebra system) to
		compute the minimal polynomial of $\gamma(0)$; it is
		$$X^3(X^2+1)(X^4+3X^2+1)(X^2+X-1)(X^2-X-1)(X-1)^3(X+1)^3.$$
		The zeros other than $0, 1, -1$ are therefore
		$$ -\varphi i,\ -(1/\varphi)i,\ (1/\varphi)i,\ \varphi i,\ 
		-\varphi,\ 1/\varphi,\ -1/\varphi,\ \varphi .$$
		It follows that $\totk(F_k)$ for $k \geq 3$ is a linear combination
		of the $k$'th powers of these zeros, together with
		$k^2,k,1,k^2(-1)^k,k(-1)^k,(-1)^k$.   We can then solve for the
		coefficients with linear algebra and hence prove both of the desired results.
	\end{proof}
	
	Furthermore, we can (almost trivially) obtain new results at will.
	As an example, consider the following result:
	\begin{proposition}
		We have $\totk(3F_n) = F_{n+2} - F_{n-4}$ for $n \geq 4$.
	\end{proposition}
	\begin{proof}
		We use the fact that $(3F_n)_F = 10^n10$, and the technique above.
	\end{proof}

	As another example of the power of these techniques, let us compute
	the average number  of Knott expansions for $n$ in the interval
	$F_i \leq n < F_{i+1}$.    The idea was already explained in
	\cite[\S 9.10]{Shallit:2023}:   the canonical Zeckendorf representations
	for $n$ in the interval $[F_i, F_{i+1})$ consist of those
	$i-1$ bit numbers that start with $1$ and have no occurrence of
	$11$.   Suppose $(v', \gamma', w')$ is a linear representation for
	$\totk$ that has been modified so that $v' \gamma' (x) w' = 0$
	if $x$ has an occurrence of $11$; then
	% \begin{align*}
		% s_\kappa(i) := \sum_{F_i \leq n < F_{i+1}} \totk(n) &=
		% \sum_{{x \text{ contains no 11 and starts with 1}} \atop {|x|=i-1}} v' \gamma'(x) w' \\
		% &= v' \gamma'(1) (\gamma'(0)+\gamma'(1))^{i-2} w' ,
		% \end{align*}
	\begin{align*}
		s_\kappa(i) := \sum_{F_i \leq n < F_{i+1}} \totk(n) &=
		\sum_{\scriptscriptstyle \substack{x \text{ contains no } 11 \\ \text{and starts with } 1 \\ |x|=i-1}} v' \gamma'(x) w' \\
		&= v' \gamma'(1) (\gamma'(0)+\gamma'(1))^{i-2} w' ,
	\end{align*}
	and hence $s_\kappa(n)$ is representable as a linear combination
	of the $i-2$'th powers of the
	zeros of the minimal polynomial for $\gamma'(0)+\gamma'(1)$.
	We can then determine the coefficients by solving a linear system.
	When we do this, we get the following result:
	\begin{theorem}
		The average number of Knott expansions for $n$ in the
		interval $[F_i, F_{i+1})$ is asymptotically
		$\Theta(\rho^i)$, where $\rho = \zeta/\varphi \doteq 1.5334624126$, and
		$\zeta \doteq 2.4811943040920156$ is the dominant zero
		of $X^3-2X^2-2X+2$.
	\end{theorem}
	
	\begin{proof}
		We use this {\tt Walnut} code to compute the linear representation:
		\begin{verbatim}
			def dek2 n "?msd_fib (~$has11(n)) & $dekking(n,x,y)":
		\end{verbatim}
		The result is a linear representation of rank $122$, with
		minimal polynomial of degree $38$.  We can use the technique of
		\cite[Theorem 6]{Fleischer&Shallit:2021}
		to refine this minimal polynomial for $s_\kappa(i)$
		to \[(X^3-2X^2-2X+2)(X+1)(X^4+3X^2 + 1)(X^2-X-1).\]
		The dominant zero here is $2.4811943040920156\cdots$ of
		$X^3-2X^2-2X+2$.  Dividing by $F_{i-1}$, the number of summands,
		gives the result for the average value.
	\end{proof}
	
	\section{Natural expansions}
	\label{four}
	
	Dekking and Van Loon also enumerated what they called ``natural'' 
	$\varphi$-representations, which are those expansions
	$n = [x.y]_\varphi$ 
	where the length of $y$ (not including trailing zeros, of course)
	is the same as the length of $y'$, where
	$(n)_\varphi = x'.y'$.
	
	We can do this in the same manner as we constructed the automaton
	{\tt saka} above: namely, we simply
	intersect the first Frougny-Sakarovitch
	automaton with another automaton that imposes the length condition.
	
	\begin{verbatim}
		reg first1match {0,1} {0,1} "[0,0]*[1,1]([0,0]|[0,1]|[1,0]|[1,1])*":
		def natural "?msd_fib Ew,x $saka(n,w,x) & $frougny(n,y,z) &
		$first1match(x,z)":
		def nat n "?msd_fib $natural(n,y,z)":
	\end{verbatim}
	which produces a linear representation of rank $123$ for $\totn(n)$, the
	number of such expansions of length $n$.
	Using this linear representation, we can easily recover 
	Theorem 4.3 of Dekking and Van Loon, as follows:
	\begin{theorem}
		We have $\totn(F_{2n+1}) = \totn(F_{2n+2}) =  F_{2n+1}$  for $n \geq 0$.
	\end{theorem}
	
	Using the same technique, we can also easily prove an additional result:
	\begin{theorem}
		We have $\totn(L_{2n+1}) = 1$ and 
		$\totn(L_{2n}) = 2n$ for $n \geq 1$.
	\end{theorem}
	
	\section{Dekking--Van Loon ``canonical'' expansions}
	\label{fiveb}
	
	Dekking and Van Loon \cite{Dekking&van.Loon:2021} introduced a different
	kind of base-$\varphi$ representation, which they called ``canonical''.
	However, to avoid confusion, in this paper, we call them DVL-expansions.
	
	In a DVL expansion, we allow $11$, but only as the digits $d_1 d_0$,
	and then only if $d_{-1} = 0$.    We can create an automaton
	{\tt dvl} that accepts DVL expansions by adding a restriction
	{\tt dvlcond} to {\tt frougny}, as depicted in Figure~\ref{dvlcond}.
	
	The input is $x$ and $y$ in parallel, corresponding to the
	$\varphi$-representation $x.y^R$, and accepts if it obeys the DVL condition.
	
	We can then accept DVL expansions, as follows:
	\begin{verbatim}
		def dvl "?msd_fib $frougny(n,x,y) & $dvlcond(x,y)":
	\end{verbatim}
	The resulting automaton has $48$ states and accepts $n,x,y$ iff
	the DVL expansion of $n$ is $[x.y^R]$.
	
	\begin{figure}[htb]
		\begin{center}
			\includegraphics[width=6in]{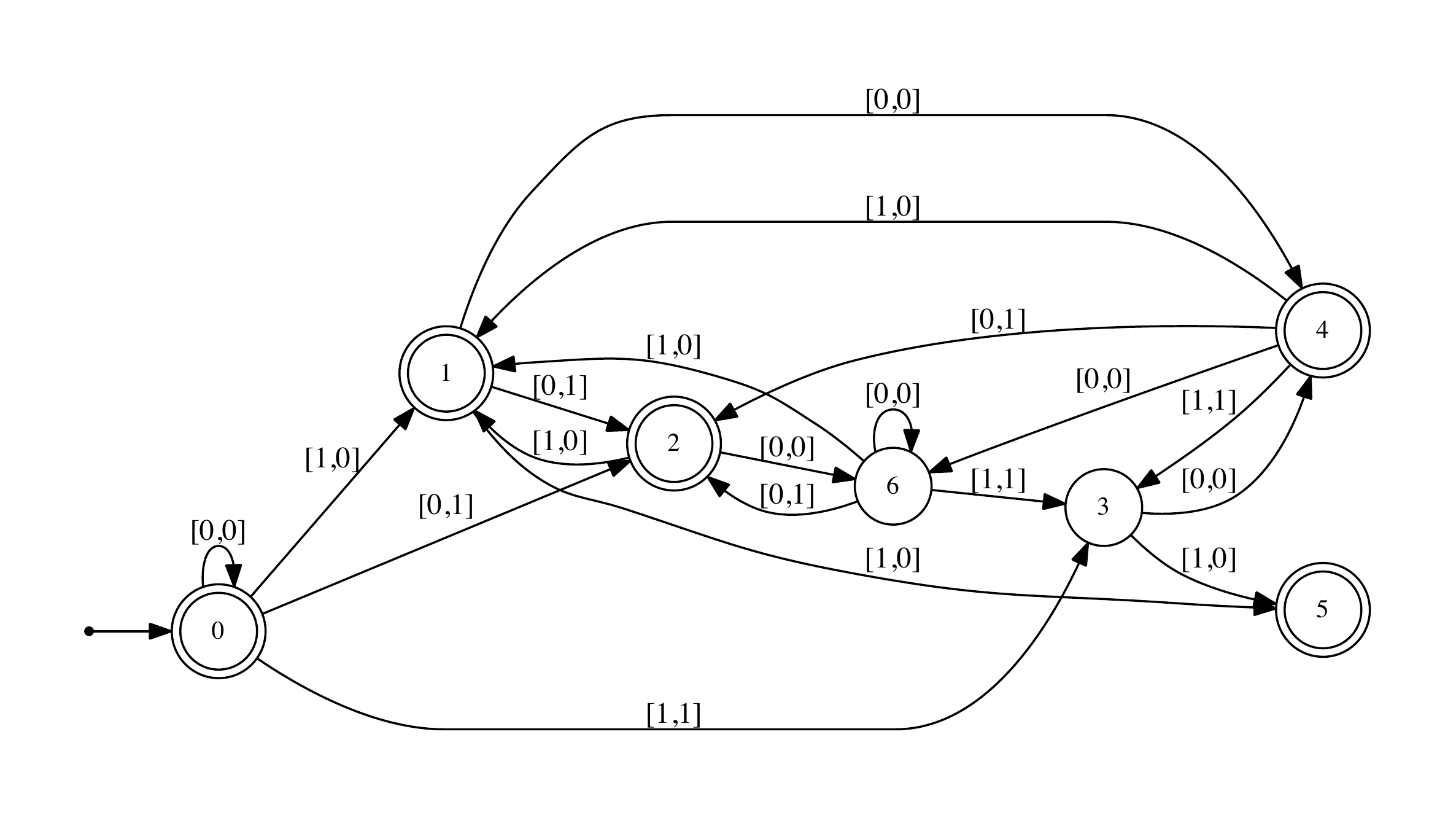}
		\end{center}
		\caption{Automaton checking the DVL condition.}
		\label{dvlcond}
	\end{figure}

	We can then use this automaton to ``automatically''
	obtain many of the results in
	\cite{Dekking&van.Loon:2021};  for example, their Proposition~3.3:
	\begin{proposition}
		The canonical $\varphi$-representation of an integer $n$ differs from the
		DVL expansion of $n$ if and only if there exists $m\geq 1$
		such that $n = \lfloor (\varphi+2)m \rfloor$.
	\end{proposition}
	\begin{proof}
		We use the {\tt Walnut} code
		\begin{verbatim}
			reg equal {0,1} {0,1} "([0,0]|[1,1])*":
			def differ "?msd_fib Ex,y,w,z $saka(n,x,y) & $dvalcanon(n,w,z) &
			((~$equal(x,w))|(~$equal(y,z)))":
			eval prop33 "?msd_fib An $differ(n) <=> (Em,t m>=1 &  $phin(m,t) &
			n=t+2*m)":
		\end{verbatim}
		and {\tt Walnut} returns {\tt TRUE}.
	\end{proof}
	
	Theorem 6.2 of Dekking and Van Loon \cite{Dekking&van.Loon:2021}
	characterized the
	lengths of vertical runs in DVL expansions; they used a case-based
	proof and 3 pages.   We can obtain their results directly using
	exactly the same ideas as for canonical expansions above in
	Theorem~\ref{vertthm}.  For example, for $i \geq 1$, we use the code
	\begin{verbatim}
		def matchdvl "?msd_fib Ex,y $dvl(n,x,y) & $isfib(t) & $match1(x,t)":
		def verticaldvl "?msd_fib (i>0) & En (~$matchdvl(n-1,t)) &
		(~$matchdvl(n+i,t)) & Aj (j<i) => $matchdvl(n+j,t)":
	\end{verbatim}
	which gives an $11$-state automaton from which we can easily read off their
	results.
	
	\section{Algebraic integers}
	\label{ten}
	
	As is well-known \cite[\S 2.4]{Stewart&Tall:1979},
	the algebraic integers of $\Que(\sqrt{5})$ are
	given by $\Zee[\varphi]$.
	The non-negative real members of $\Zee[\varphi]$ are precisely
	those with finite $\varphi$-representations \cite{Rousseau:1995}.  
	
	Given a real number of the form $z = m\varphi + n \geq 0 $ with $m, n \in \Zee$,
	we would like to compute its $\varphi$-representation with an automaton.
	Since $m, n$ could each possibly be negative, it makes sense to represent
	them in negaFibonacci representation.   We can then compute the
	appropriate automaton, using Eqsuations~\eqref{berg1} and \eqref{berg2}
	by a small variation of what we did in Section~\ref{two}.  
	We give the {\tt Walnut code}:
	\begin{verbatim}
		def intpartleft "?msd_fib Er,s,t,y,b,c $shiftr(x,r) & $shiftr(r,s) &
		$shiftr(s,t) & $fibnorm(t,y) & $lstbit1(x,b) & $lstbit3(x,c) &
		z=y+b+c":
		
		def phipartleft "?msd_fib Er,s,y,b $shiftr(x,r) & $shiftr(r,s) &
		$fibnorm(s,y) & $lstbit2(x,b) & z=y+b":
		
		def intpartright "?msd_neg_fib Er $shiftl(x,r) & $negfibnorm(r,z)":
		
		def phipartright "?msd_neg_fib $negfibnorm(x,z)":
		
		def genfrou "?msd_neg_fib Ez,t,u,q,r,s $intpartleft(x,?msd_fib z) &
		$fibnegfib((?msd_fib z),t) & $intpartright(y,u) & n=t+u &
		$phipartleft(x,?msd_fib q) & $phipartright(y,r) & 
		$fibnegfib((?msd_fib q),s) & m=r+s":
		# 536 states
		
		def canfrou "?msd_neg_fib $genfrou(m,n,x,y) & $no11xy(x,y)":
		# 259 states
	\end{verbatim}
	Here {\tt genfrou} does not assume that the expansion given
	by $[x.y^R]$ is canonical, but {\tt canfrou} does.
	
	As an application, let us find the algebraic numbers 
	$m \varphi + n$ such that the left and right parts of their
	base-$\varphi$ representation have the same length.
	
	\begin{theorem}
		There is a $96$-state automaton that accepts those pairs
		$m,n$, in negaFibonacci representation,
		such that $(m \varphi + n)_\varphi$ has left and right parts
		of the same length.
	\end{theorem}
	
	\begin{proof}
		We use the following {\tt Walnut} code:
		\begin{verbatim}
			def samelen "?msd_neg_fib Ex,y $canfrou(m,n,x,y) & $first1match(x,y)":
		\end{verbatim}
	\end{proof}

	\section{A final word}
	\label{nine}
	
	{\tt Walnut} is available for free download at \\
	\centerline{\url{https://cs.uwaterloo.ca/~shallit/walnut.html}.} 
	For the definitions of {\tt Walnut} automata that we did not
	give explicitly, see  \\
	\centerline{\url{https://cs.uwaterloo.ca/~shallit/papers.html}.} 
	
	Most of the other results of \cite{Dekking&van.Loon:2023} can be
	proved using the analogous techniques.
	
	In principle, everything we have said in this paper can be applied
	to $\beta$-expansions, where $\beta$ is a quadratic Pisot number.
	The choice $\beta = \varphi$ was particularly easy in the current
	version of {\tt Walnut}, because it already
	has Fibonacci and negaFibonacci numbers
	implemented, and adders for integers represented in these forms.
	
	% A table of contents will be automatically inserted in your article if it
	% has 3 or more sections.  Please, do not try to manually change this
	% behaviour.
	
	% Also, please consider the following suggestions while preparing your 
	% manuscript (as they will speed up the editorial process):
	% * Avoid starting a new sentence with a mathematical formula;
	% * Try to separate adjacent formulas with words;
	% * Avoid inline formulas longer than half of a line. You can use math 
	%   displays (\[...\]) instead;
	% * Consider the use the enumerate and itemize environments for lists;
	% * Consider the use of \dots, \ldots, \dotsc, \cdot, etc, instead of "..." 
	%   or ".";
	% * Instead of numbering or citing an article by hand (using parenthesis or 
	%   brackets), consider the use of \cite, \ref and \eqref for citations and
	%   cross-references;
	% * Try to avoid inserting horizontal or vertical spacing, such as \hskip, 
	%   \vskip and \bigskip;
	% * Try to avoid inserting line or page brakes, such as \\, \newpage and
	%   \clearpage.
	
	% Acknowledgments should be added at the end of this section (right before
	% the refences section) as a \subsection* (a subsection without a number):
	\subsection*{Acknowledgments} 
	I thank Anna Lubiw, George Bergman, Jacques Sakarovitch,
	Michel Dekking, and Daniel Berend for helpful comments and insights.
	
	I also thank the two anonymous referees, who read the paper carefully
	and provided many useful suggestions.
	
	Finally, I thank both George Bergman (happy 80th birthday, George!) and 
	Christiane Frougny (Joyeux anniversaire, Christiane !) for their inspiring work over many decades.
	
	%%% REFERENCES %%%
	%{\small\bibliography{Shallit}}
	% Please, do not change the above line and do not insert your references
	% into this file.  Instead, insert your references into the cimart.bib file.
	% See cimart.bib for further instructions.

	\EditInfo{December 1, 2023}{April 25, 2024}{Rigo Michel, Emilie Charlier and Julien Leroy}

\end{document}